\documentclass[twoside]{article}

\usepackage[preprint]{aistats2026}
\usepackage{hyperref}
\usepackage[round]{natbib}
\usepackage{caption}
\usepackage{subcaption}
\usepackage{float}

\newcommand*{\argmax}{\mathop{\mathrm{argmax}}}
\newcommand*{\argsup}{\mathop{\mathrm{argsup}}}

\usepackage{url}            %
\usepackage{booktabs}       %
\usepackage{amsfonts}       %
\usepackage{nicefrac}       %
\usepackage{microtype}      %
\usepackage{xspace}
\usepackage{amsmath}
\usepackage{amssymb}
\usepackage{mathtools}
\usepackage{amsthm}
\usepackage{thmtools}  
\usepackage{algorithm,algorithmic}
\usepackage{color}
\usepackage{enumitem}
\usepackage{comment}
\usepackage{bm}
\usepackage{graphicx}
\usepackage{wrapfig}
\usepackage[size=small]{todonotes}
\newcommand{\todoks}[2][] {\todo[author=KS:,color=green!40, #1]{#2}}

\usepackage{algorithm}
\usepackage[ruled,linesnumbered,algo2e]{algorithm2e}

\usepackage{natbib}
\bibpunct{(}{)}{;}{a}{,}{,}

\newcommand{\supp}{\operatorname{supp}}
\newcommand{\vol}{\operatorname{vol}}
\newcommand{\diam}{\operatorname{diam}}
\newcommand{\inner}{\operatorname{inner}}
\newcommand{\out}{\operatorname{outer}}

\newcommand{\RR}{\mathbb{R}} 
\newcommand{\E}{\mathbb{E}}
\newcommand{\Xcal}{\mathcal{X}}
\newcommand{\Scal}{{\mathcal{S}}}
\newcommand{\Mcal}{{\mathcal{M}}}
\newcommand{\Pcal}{{\mathcal{P}}}
\newcommand{\Fcal}{{\mathcal{F}}}
\newcommand{\Chat}{{\hat{C}}}
\newcommand{\Ccheck}{{\check{C}}}
\newcommand{\rhat}{{\hat{r}}}
\newcommand{\rcheck}{{\check{r}}}
\newcommand{\rbar}{{\bar{r}}}
\newcommand{\fhat}{{\hat{f}}}
\newcommand{\rtilde}{{\tilde{r}}}
\newcommand{\xhat}{{\hat{x}}}
\newcommand{\Fhat}{{\hat{F}}}
\newcommand{\Fcheck}{{\check{F}}}
\newcommand{\kl}{{\texttt{KL}}}

\DeclareMathOperator*{\ind}{1{\hskip -2.5 pt}\hbox{I}}  
\newcommand{\EE}{\mathbb{E}} 
\DeclareMathOperator*{\Var}{{\rm Var}}
\newcommand{\pr}{\mathop{\mathrm{Pr}}}  
\def\iid{{\it i.i.d.}}

\usepackage[capitalize,noabbrev]{cleveref}

\theoremstyle{plain}
\newtheorem{theorem}{Theorem}[section]

\newtheorem{lemma}[theorem]{Lemma}
\newtheorem{corollary}[theorem]{Corollary}
\theoremstyle{definition}
\newtheorem{definition}[theorem]{Definition}
\newtheorem{assumption}[theorem]{Assumption}
\theoremstyle{remark}
\newtheorem{remark}[theorem]{Remark}

\begin{document}

%

%

\twocolumn[

\aistatstitle{Optimal rates for density and mode estimation with expand-and-sparsify representations}

\aistatsauthor{ Kaushik Sinha$^*$ \And Christopher Tosh$^*$}

\aistatsaddress{ Wichita State University \\ Wichita, Kansas, USA \\
\href{mailto:kaushik.sinha@wichita.edu}{kaushik.sinha@wichita.edu}\And  Memorial Sloan Kettering Cancer Center \\ New York City, New York, USA \\
\href{mailto:christopher.j.tosh@gmail.com}{christopher.j.tosh@gmail.com}} ]

\def\thefootnote{*}\footnotetext{Both  authors contributed equally to this work.}

\begin{abstract}
Expand-and-sparsify representations are a class of theoretical models that capture sparse  representation phenomena observed in the sensory systems of many animals. At a high level, these representations map an input $x \in \RR^d$ to a much higher dimension $m \gg d$ via random linear projections before zeroing out all but the $k \ll m$ largest entries. The result is a $k$-sparse vector in $\{0,1\}^m$. We study the suitability of this representation for two fundamental statistical problems: density estimation and mode estimation. For density estimation, we show that a simple linear function of the expand-and-sparsify representation produces an estimator with  minimax-optimal $\ell_{\infty}$ convergence rates. In mode estimation, we provide simple algorithms on top of our density estimator that recover single or multiple modes at optimal rates up to logarithmic factors under mild conditions.
\end{abstract}

\section{Introduction}
\label{sec:intro}

Neural circuits that transform dense, lower-dimensional signals to sparse, higher-dimensional signals are a recurring motif in animal sensory systems. In the olfactory system of Drosophila, for example, a dense signal consisting of 50 glomeruli in the antennal lobe is lifted, via random projections, to an approximately 2000-dimensional signal where only the top 5\% of coordinates are non-zero~\citep{turner2008, caron2013}. Similar phenomena are observed in the olfactory system of rodents, the visual cortex of cats, and the electrosensory system of electric fish~\citep{stettler2009,olshausen2004,chacron2011}. The widespread nature of these transformations suggests they are useful in representing certain types of data.

Expand-and-sparsify representations are a class of mathematical formalisms to capture and analyze these ubiquitous neural phenomena~\citep{dasgupta2020expressivity}. At a high level, an expand-and-sparsify representation consists of a random projection into a higher-dimensional space followed by a sparsification operation. Recent theoretical work has shown that these expand-and-sparsify representations are well-adapted to a variety of statistical applications, including locality-sensitive hashing~\citep{dasgupta:etal:2017:exp-spars-lsh}, count sketching~\citep{dasgupta:etal:2022:exp-spars-count-sketch}, function approximation~\citep{dasgupta2020expressivity}, and non-parametric classification~\citep{sinha2024_nonparametric}.

In this work, we identify another fundamental statistical task that expand-and-sparsify representations are well-suited for: non-parametric density estimation. Specifically, we show that given $n$ \iid ~samples 
from a smooth density $f$ whose support is restricted to $\Scal^{d-1}$, there is an estimator $\fhat_n$ that is a linear function of the expand-and-sparsify representation that achieves minimax-optimal convergence rates in $\ell_\infty$ distance. Our bounds hold with high probability, simultaneously for all $x$ in the support of $f$. 
Thus, this paper bridges two important areas of study: sparse neural representations of data and density estimation. The expand-and-sparsify representation is common in nature, but its properties are poorly understood. By showing what types of tasks it can solve, we can gain insight on what kinds of circuits might be present in nature. On the other hand, density estimation is a fundamental statistical task that has been studied for decades. Previous theoretical work in this area has been dominated by k-NN and kernel-based methods. Our work shows a completely different approach based on sparse random representations achieves optimal convergence rates.

Beyond density estimation, we present simple algorithms, utilizing the expand-and-sparsify representation density estimator $\fhat_n$, for another fundamental statistical task, namely, non-parametric mode estimation. 
For a unimodal density, this algorithm simply returns the maximizer of $\fhat_n$ over the finite sample that was used to construct $\fhat_n$. We show that with high probability, the $\ell_2$ distance between the mode and the estimated mode decays at an optimal rate (up to logarithmic factors). For a density having multiple unknown modes, we present a simple iterative algorithm that outputs a set of candidate modes by taking the maximizer of $\fhat_n$ over appropriately chosen level sets.
We show that for any mode of $f$ that is well separated from the rest of the modes by low density regions we can always find, with high probability, a mode from the set of candidate modes returned by our algorithm whose $\ell_2$ distance to this well-separated mode decays at an optimal rate (up to logarithmic factors).

We summarize our main contributions below:
\begin{itemize}
    \item We demonstrate the suitability of expand-and-sparsify representation for two fundamental statistical problems: density estimation and mode estimation. 
    \item For density estimation, we show that our proposed density estimator converges at a minimax-optimal rate $O\left(\left(\frac{\log n}{n}\right)^{\frac{\beta}{2\beta+{d-1}}}\right)$ in $\ell_{\infty}$ distance, where $\beta$ is the H\"{o}lder smoothness constant of the distribution to be approximated.
    \item For mode estimation, we show that under mild conditions, our proposed simple algorithms on top of our estimated density can recover modes at an optimal rate $\tilde{O}\left(n^{-\frac{1}{d+3}}\right)$ (up to logarithmic factors).
\end{itemize}

The rest of the paper is organized as follows. We discuss related work in \cref{sec:related_works}. We summarize various notations used in this paper in \cref{sec:notations}. We present our proposed density estimator using expand-and-sparsify presentation in \cref{sec:estimator} and derive density estimation rates in \cref{sec:density_rates}. We present our mode estimation results in \cref{sec:mode_estimation}, empirical evaluations in \cref{sec:empirical}, and conclude in \cref{sec:conclusions}.
\section{Related work}\label{sec:related_works}

\paragraph{Density estimation.} Non-parametric density estimation is a well-studied problem with known optimal rates. The estimator considered in the present paper bears a strong resemblance to the well-known $k$-nearest neighbor (kNN) estimator and the kernel density estimator (KDE), both of which are known to achieve minimax optimal rates under general conditions~\citep{dasgupta2014_knn_density, rinaldo:wasserman:2010:density-clustering, jiang2017_kde_convergence}. Specifically, in $d$-dimensions, both the kNN estimator and KDE achieve convergence rates of $\Tilde{O}\left( n^{-\frac{\beta}{2\beta + d}}\right)$ in $\ell_\infty$ distance, where $\beta$ is the H\"{o}lder smoothness constant of the distribution to be approximated.

One interesting way in which the expand-and-sparsify estimator differs from standard density estimators such as kNN and KDE is its use of randomization. Note that once a dataset and a hyperparameter have been fixed, the KDE and kNN estimators are deterministic. However, the expand-and-sparsify density estimator has an additional source of randomness in its choice of high-dimensional projection. This extra randomness is a property shared with a recently proposed random forest density estimator that builds a set of random $k$-d trees that partition the space and count the number of points that fall into the leaf nodes~\citep{wen2022rf_density}. That work demonstrated that the estimator has expected convergence rate $\Tilde{O}\left( n^{-\frac{1 - 4^{-\beta}}{1 - 4^{-\beta} + d \log 2}}\right)$ in $\ell_2$ distance.

\paragraph{Expand-and-sparsify representations.} As discussed in the introduction, there have been many recent theoretical works exploring the utility of expand-and-sparsify representations for statistical tasks, ranging from classification~\citep{sinha2021fruitfly,ram2022federated,sinha2024_nonparametric} to locality-sensitive hashing~\citep{dasgupta:etal:2017:exp-spars-lsh}.

Of particular relevance to the present paper is the work of~\cite{dasgupta2020expressivity}, which showed that for any smooth function, there is a linear function of the corresponding expand-and-sparsify representation that achieves low $\ell_\infty$ error. While this result demonstrates the existence of a good expand-and-sparsify density estimator in principle, it does not provide an algorithm to construct one from a finite sample. Indeed, as an unsupervised learning problem, it is unclear a priori how to estimate the corresponding coefficients from data. This is a core contribution of the present paper.

\paragraph{Mode estimation.} 
There is an extensive literature on mode estimation, and we review some of the relevant works here. A classical approach to mode estimation consists of estimating the single mode $x_0$ of $f$ as $\hat{x}_0 = \argsup_{x\in\RR^d}f_n(x)$ where $f_n$ is an estimate of $f$ (the kernel density estimator is a typical choice for $f_n$). 
Note that this mode estimator is generally hard to implement in practice. Nevertheless, a series of works on this approach \citep{parzen1962_density_and_mode,chernoff1964_mode,eddy1980_mode_estimator} establish consistency as well as as convergence rates for this procedure under various regularity assumptions on the distribution. 
More recent works \citep{grund1995_mode_estimation,klemela2005_mode_estimation,tsybakov1990_mode,donoho1991_rates_of_convergence} address the problem of optimal choice of kernel bandwidth to adaptively achieve minimax risk for mode estimation estimation. 
In particular, for a $\kappa$-times differentiable $f$, \citet{tsybakov1990_mode,donoho1991_rates_of_convergence}
 independently established minimax risk of the form $n^{-\frac{\kappa-1}{2\kappa+d}}$. 
 Another line of work estimates the mode from practical statistics of the data \citep{grenander1965_direct_mode_estimation,abraham2010_mode_estimate,dasgupta2014_knn_density}. 
 In particular, \cite{abraham2010_mode_estimate,dasgupta2014_knn_density} use a simple and practical mode estimator $\argmax_{x\in X_n}f_n(x)$, where $X_n = \{x_1,\ldots,x_n\}$ are sampled \iid\,
 according to $f$. When $f_n$ is the kernel density estimator, \cite{abraham2010_mode_estimate} shows that this mode estimator is consistent. 
 On the other hand, when $f_n$ is the kNN density estimator, \cite{dasgupta2014_knn_density} shows that this mode estimator is not only consistent but also converges at a minimax-optimal rate. The mode estimation results in the present paper are similar to the results presented in \cite{dasgupta2014_knn_density} when $f$ is restricted to $\Scal^{d-1}$.

 When the distribution is not unimodal, the best known practical approach for estimating all the modes is the mean-shift algorithm and its variations \citep{cheng1995_mean_shift,fukunaga1975_density_gradient,comaniciu2002_mean_shift,li2007_clustering_via_mode_estimation,arias-castro2016_consistency_of_mean_shift}, which essentially consists of gradient ascent of a sufficiently smooth $f_n$ starting from every sample point. 
 Under the milder condition of requiring only that $f$ is well-approximated by a quadratic in a neighborhood of each mode, \cite{dasgupta2014_knn_density} obtained a finite sample minimax-optimal rate on $\|\hat{x}_0-x_0\|$, where $\hat{x}_0$ is estimate of any mode $x_0$, of the form $O\left(n^{-\frac{1}{4+d}}\right)$ that holds with high probability. Again, our multiple mode estimation results in the present paper are similar to the results presented in \cite{dasgupta2014_knn_density} when $f$ is restricted to $\Scal^{d-1}$.

\section{Notations}\label{sec:notations}
Throughout this paper, we assume access to $X_n =\{x_1,\ldots,x_n\}$ drawn \iid ~from a distribution $\Fcal$ over $\Scal^{d-1}$, with Lebesgue density function $f$. For any $A\subset\Scal^{d-1}$, we use $f(A) = \int_{A}f(x)dx$ to denote the probability mass of $A$ under $f$. We use $\|\cdot\|$ to denote $\ell_2$ norm and $\|\cdot\|_{\infty}$ to denote $\ell_{\infty}$ norm. For any $x\in\Scal^{d-1}$ and $r>0$, we use $B(x,r) = \{x' \in\Scal^{d-1}: \|x-x'\|\leq r\}$ to denote the closed ball of radius $r$ and centered at $x$, that intersects with $\Scal^{d-1}$. For any $A\subset\Scal^{d-1}$ and $r>0$, we write $A_r = \{x\in\Scal^{d-1} : \inf_{y\in A} \|y-x\|\leq r\}$. 
We use $\tilde{O}\left(\cdot\right)$ to denote the `big-Oh' notation that hides logarithmic factors. For clear exposition, most of our proofs are deferred to the supplementary material which is organized section-wise. However, when appropriate, we provide proof sketches for important theoretical results.
\section{Expand-and-sparsify density estimation}
\label{sec:estimator}

The expand-and-sparsify representation of an input $x \in \Xcal \subset \RR^d$ is given by the following transformation (Fig. \ref{fig:eas_representation}).
\begin{enumerate}
    \item A high-dimensional linear mapping, $y = \Theta x \in \RR^m$.
    \item A sparsification scheme to a binary vector $z\in \{0,1\}^m$, where \[z_j = \begin{cases} 1 & \text{ if $y_j$ is one of $k$ largest entries in $y$} \\ 0 & \text{ otherwise} \end{cases}\]
\end{enumerate}
Here, $\Theta \in \RR^{m \times d}$ is a random matrix with rows $\theta_1, \ldots, \theta_m$ drawn i.i.d. from some measure $\nu$. We denote the mapping $x \mapsto z$ via the function $h(x) \in \{0,1\}^m$, where each coordinate $j$ is indexed as $h_j(x) = z_j$. 


\begin{figure}
  \begin{center}
    \includegraphics[scale=0.3]{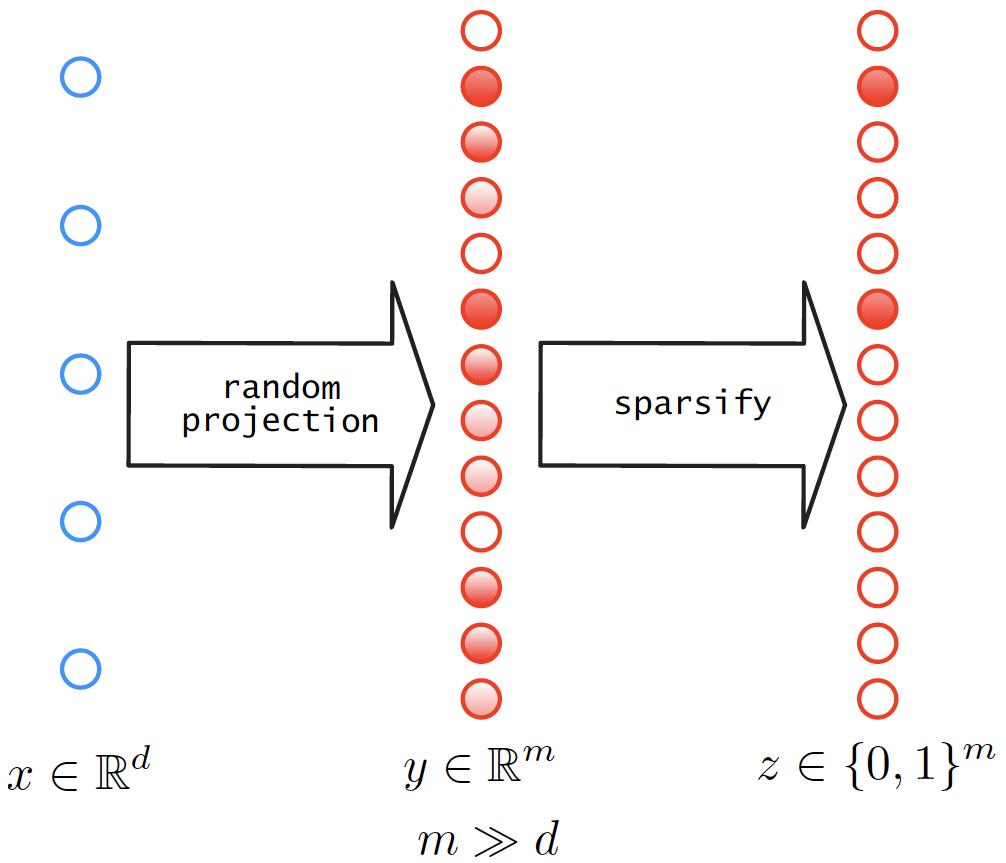}
  \end{center}
  \caption{Visualization of the expand-and-sparsify architecture \citep{dasgupta2020expressivity}.}
  \label{fig:eas_representation}
\end{figure}

Observe that for any vector $x \in \RR^d$, we have $h(c x) = h(x)$ for any $c > 0$. Thus, we will restrict ourselves to the setting where elements of $\Xcal$ have unit length, i.e. $\Xcal \subset \Scal^{d-1} = \{ x \in \RR^d \, : \, \| x \| = 1\}$. We will also focus on the case where $\nu$ is uniform over $\Scal^{d-1}$.

Let $C_j = \{x \in \Xcal \, : \, h_j(x) = 1\}$ denote the inputs satisfying that the $j$-th coordinate in the expand-and-sparsify representation is set to 1. Since $\Xcal$ and $\supp(\nu)$ are both subsets of $\Scal^{d-1}$, $C_j$ consists exactly of those inputs $x$ whose $k$ closest elements of $\theta_1, \ldots, \theta_m$ contains $\theta_j$.

Let $f$ be some density over $\Xcal$, and suppose $x_1,\ldots,x_n$ are drawn \iid ~from $f$. Let $f_{n}(C_j) = \frac{1}{n}\# \{ i : x_i \in C_j \} $ denote the fraction of sample points whose $j$th coordinate is set to 1. When the measure $\nu$ is uniform over $\Scal^{d-1}$, the expand-and-sparsify density estimator is given by
\begin{align}\label{eqn:eas_density_estimator}
\hat{f}_n(x) 
&= \frac{m}{k^2 S_{d-1}} \sum_{j=1}^m f_{n}(C_j) h_j(x)  \nonumber \\
& = \frac{m}{k^2 S_{d-1}} \sum_{j: h_j(x) = 1} f_{n}(C_j),
\end{align}
where $S_{d-1} = \frac{2\pi^{d/2}}{\Gamma(d/2)}$ is the surface area of $\Scal^{d-1}$.

To see where this estimator comes from, consider the following (computationally intractable\footnote{$C_j$ is a subset of a high-dimensional hypersphere that is only implicitly specified, and computing the volume of a high-dimensional set with a membership oracle is a known hard problem \citep{fuerdi1986_volume}.}) estimator:
\[ \tilde{f}_n(x) = \frac{1}{k} \sum_{j=1}^m \frac{f_{n}(C_j)}{\vol(C_j)} h_j(x) = \frac{1}{k} \sum_{j: h_j(x)=1} \frac{f_{n}(C_j)}{\vol(C_j)}, \]
where $\vol(C_j)$ is the $d-1$-dimensional volume of $C_j$, i.e. the corresponding hypersurface area. In this light, $f_{n}(C_j)$ is an estimate of the integral $f(C_j) = \int_{C_j} f(x) \, dx$, and normalizing this value by $\vol(C_j)$ gives us an estimate of $f(u)$ for $u$ drawn uniformly from $C_j$. If $f$ is smooth, $C_j$ is small in diameter, and $x \in C_j$, then $f_{n}(C_j)/\vol(C_j)$ will be close to $f(x)$. Averaging over all $j$ such that $x \in C_j$ gives us the estimator $\tilde{f}$. Moving from $\tilde{f}_n$ to $\hat{f}_n$ is a matter of showing that $\vol(C_j) \approx \frac{k}{m} S_{d-1}$ when $\nu$ is uniform over $\Scal^{d-1}$, which we demonstrate in \cref{eqn:volume-approximation}, below.
\section{Density estimation rates}
\label{sec:density_rates}

In this section, we analyze the convergence rate of the expand-and-sparsify estimator $\hat{f}_n$. We first give a result for general densities with no smoothness assumptions, before specializing this result to a class of generalized H\"{o}lder smooth densities. Finally, we give a concrete matching lower bound for density estimation on the sphere, demonstrating that the expand-and-sparsify estimator is minimax optimal.

\subsection{General density estimation rates}

Our result for estimating general densities will be framed using the following notions of modulus of continuity that captures how much $f$ changes in a neighborhood of any $x$.

\begin{definition}
\label{def:modulus_of_continuity}
For $x\in\Scal^{d-1}$ and $\epsilon>0$, define $\rhat(\epsilon,x)\triangleq \sup\left\{r: \sup_{\|x-x'\|\leq r}f(x')-f(x)\leq \epsilon\right\}$, and $\rcheck(\epsilon,x)\triangleq \sup\left\{r: \sup_{\|x-x'\|\leq r}f(x)-f(x')\leq \epsilon\right\}$.
\end{definition}

Our first observation is that with high probability each region $C_j$ is sandwiched between two balls of nearly equal volume.

\begin{lemma}
\label{lem:inner-outer-balls}
If $\delta > 0$ and $m > k > C_\delta^2 d \log m$, then with probability at least $1-\delta$ the following holds. For every $j=1,\ldots,m$:
\[ B(\theta_j, r_{\inner}) \subset C_j \subset B(\theta_j, r_{\out}), \] 
where $C_\delta = c_0 \log(2/\delta)$ for some absolute constant $c_0$ and
\begin{align*}
\vol( B(\theta_j, r_{\inner}) ) &= \frac{k}{m} S_{d-1} \left(1 - C_\delta \sqrt{\frac{d \log m}{k}} \right) \\
\vol( B(\theta_j, r_{\out}) ) &= \frac{k}{m} S_{d-1} \left(1 + C_\delta \sqrt{\frac{d \log m}{k}} \right).
\end{align*}
\end{lemma}

An immediate consequence of \cref{lem:inner-outer-balls} is that with high probability, each region $C_j$ satisfies $\vol(C_j) \approx S_{d-1} k/m$. Specifically, we have
\begin{align}
\label{eqn:volume-approximation}
1 - C_\delta \sqrt{\frac{d \log m}{k}} \leq \frac{\vol(C_j)}{S_{d-1} k/m } \leq 1 + C_\delta \sqrt{\frac{d \log m}{k}}. 
\end{align}

Our second observation is that the empirical $f_n$-mass of each $C_j$ will tightly approximate its true $f$-mass, essentially a consequence of Bernstein's inequality.

\begin{lemma}
\label{lem:empirical_mass_C_j}
Suppose $\delta > 0$. With probability at least $1-\delta$,
\[ \left| f_n(C_j) - f(C_j) \right| \leq \sqrt{\frac{2{f(C_j)\log(m/\delta)}}{n}}+\frac{2\log(m/\delta)}{3n} \]
for all $j = 1,\ldots,m$.
\end{lemma}

Given these two results, we have the following general approximation result.

\begin{lemma}\label{lem:relative_density_bound}
Pick $\delta>0, k\geq 4C_{\delta}^2d\log m$ and assume that $\nu$ is uniform over $\Scal^{d-1}$. Let $\gamma_n = \alpha_n/(S_{d-1}(k/m))$, where $\alpha_n = 2\sqrt{\frac{kS_{d-1}\|f\|_{\infty}\log(m/\delta)}{mn}}+\frac{2\log (m/\delta)}{3n}$. If the conclusions of \cref{lem:inner-outer-balls} and \cref{lem:empirical_mass_C_j} hold, the following is true for all $x\in\Scal^{d-1}$ and $\epsilon>0$,
\[(i) \, \fhat_n(x)\leq\left(1+C_{\delta}\sqrt{\frac{d\log m}{k}}\right)(f(x)+\epsilon)+\gamma_n\]
\[(ii) \, \fhat_n(x)\geq\left(1-C_{\delta}\sqrt{\frac{d\log m}{k}}\right)(f(x)-\epsilon)-\gamma_n\]
provided $\diam(C_j)\leq \min\{\rcheck(\epsilon,x), \rhat(\epsilon,x)\}$ for every $j=1,\ldots,m$.
\end{lemma}

\subsection{Smooth density estimation rates}

Moving to the smooth setting, we will prove convergence rates for the expand-and-sparsify density estimator for densities satisfying the following generalization of H\"{o}lder smoothness.

\begin{definition}
\label{def:holder-smoothness}
We say that $f:\Xcal \rightarrow \RR$ is $L,\beta, r$-smooth if for all $x, x' \in \Xcal$ satisfying $\|x- x'\| \leq r$, we have $|f(x) - f(x')| \leq L \| x- x'\|^\beta$. We say $f$ is $L, \beta$-smooth if it is $L,\beta, r$-smooth for any $r>0$.
\end{definition}

Our analysis here is based on applying \cref{lem:relative_density_bound}, which shows that so long as the diameters of the regions $C_j$ are sufficiently small, then $\hat{f}_n$ provides an accurate estimate of $f$. The following result, which follows readily from \cref{lem:inner-outer-balls}, provides bounds on these diameters.

\begin{lemma}
\label{lem:general-diameter-bound}
Let $k \geq C_\delta^2 d \log m$. If the conclusion of \cref{lem:inner-outer-balls} holds, then
\[  \diam(C_j) \leq \frac{4}{\sqrt{3}} \left( \frac{6 \sqrt{d} k}{m}\right)^{1/(d-1)} \]
for all $j=1,\ldots,m$.
\end{lemma}

Combining \cref{lem:relative_density_bound,lem:general-diameter-bound} gives us the following result.

\begin{theorem}
\label{thm:smooth-density-bound}
Say $f$ is $L,\beta, r$-smooth and $\nu$ is uniform over $\Scal^{d-1}$. Pick $\delta > 0$, $k \geq \max\{4 C_\delta^2 d \log m, L^{-2} r^{-2\beta}\}$, and $m \geq 6\sqrt{d} \left( \frac{4}{\sqrt{3}} \right)^{d-1} L^{(d-1)/\beta} k^{1 + \frac{d-1}{2\beta}}$. Then with probability at least $1-2\delta$, 
\[ \|f - \hat{f}_n \|_\infty \leq \left( \|f\|_\infty C_\delta \sqrt{d \log m} + 2 \right) k^{-1/2} + \gamma_n ,\]
where $\gamma_n$ was defined in \cref{lem:relative_density_bound}.
\end{theorem}
To give an intuition on how the above pieces fit together, we provide here a short sketch of the proof of \cref{thm:smooth-density-bound}.
\begin{proof}[Proof sketch]
By \cref{def:modulus_of_continuity,def:holder-smoothness}, we have $\rhat(\epsilon,x), \rcheck(\epsilon, x) \geq \left( \frac{\epsilon}{L}\right)^{1/\beta}$ for all $x \in \Scal^{d-1}$ and $\epsilon > 0$. Applying \cref{lem:general-diameter-bound} with large enough $m$, we can then guarantee that
\[ \diam(C_j) \leq \left( \frac{\epsilon}{L}\right)^{1/\beta} \leq \min\{\rhat(\epsilon,x), \rcheck(\epsilon, x)\}. \]
With this bound in hand, we can then apply \cref{lem:relative_density_bound} to finish the proof.
\end{proof}

Instantiating \cref{thm:smooth-density-bound} with $m = \Theta(n)$ and $k = \Theta\left( m^{\frac{2\beta}{2\beta+(d-1)}}\left(\log m\right)^{\frac{d-1}{2\beta+(d-1)}}\right)$ leads to 
\[ \sup_{x\in\Scal^{d-1}}|f(x)-\fhat_n(x)|  = O\left(\left(\frac{\log n}{n}\right)^{\frac{\beta}{2\beta+(d-1)}}\right).\] 
As we will show in the next section, this rate is in fact minimax optimal.

\subsection{Lower bounds on density estimation rates}
Next, we derive a matching density estimation lower bound for any $L, \beta$-smooth density, whose support is restricted to $\Scal^{d-1}$. Our main result is as follows.
\begin{theorem}\label{thm:density_lower_bound}
Let $\mathcal{P}$ be the class of probability  distributions whose densities are supported on $\mathcal{S}^{d-1}$ and are $L,\beta$ smooth. For any $P\in\mathcal{P}$, we denote its density by $f_P$. Let $x_1,\ldots, x_n$ be a sample from some distribution $P\in\mathcal{P}$ and let $\hat{f}_n = \hat{f}_n(x_1,\ldots,x_n)$ be an estimator of $f_P$. Then, for $d>4$ and $n\geq 16$, the following holds.
\[\inf_{\hat{f}_n}\sup_{P\in\mathcal{P}}\mathbb{E}_P\left(\|\hat{f}_n-f_P\|_{\infty}\right) = \Omega\left(\left(\frac{\log n}{n}\right)^{\frac{\beta}{2\beta+(d-1)}}\right).\]
\end{theorem}
We present a short sketch of the proof here. It is structurally similar to other lower bound proofs for smooth density estimation in $\RR^d$, e.g. \citep{khasminskii1979lowerbound, tsybakov:2008:nonparametric}, but adapted here to the setting of the sphere.
\begin{proof}[Proof sketch]
The high level idea is to find a packing of the unit sphere and construct a finite set of smooth functions, using this packing, with non-overlapping support. We then construct a finite family of $L,\beta$ smooth densities using the functions from above. We can show that these densities are close to each other in KL divergence but far from each other in $\ell_\infty$ distance. An application of Fano's inequality~\citep{tsybakov:2008:nonparametric} finishes the proof. 
\end{proof}


\section{Mode estimation}\label{sec:mode_estimation} 
We now present analyses of two simple algorithms that use the expand-and-sparsify representation for estimating modes of an unknown density. We start with the following definition of modes. 
\begin{definition}
    We denote the set of modes of a density $f$ by: \[\Mcal = \{x: \exists r>0, \forall x'\in B(x,r),f(x')<f(x)\}.\]
\end{definition}
We follow the general approach of \citep{dasgupta2014_knn_density} for our mode estimation results and make the following assumption.
\begin{assumption}\label{assumption:differentiability}
$f$ is twice differentiable in a neighborhood of every $x\in\Mcal$. We denote the gradient and Hessian of $f$ by $\nabla f$ and $\nabla^2f$. Furthermore, $\nabla^2f(x)$ is negative definite at all $x\in\Mcal$.
\end{assumption}
Assumption \ref{assumption:differentiability}, like most previous work, considers only interior modes and excludes modes at the boundary of the support of $f$, where it is not continuously differentiable. 
A direct implication of Assumption \ref{assumption:differentiability} is that for any $x\in\Mcal, \nabla f$ is continuous in a neighborhood of $x$, with $\nabla f(x)=0$. Additionally, negative definiteness of $\nabla^2 f(x)$ at $x$ allows $f$ to be well approximated by a quadratic in a neighborhood of $x$ as detailed below.
\begin{lemma}
    [\cite{dasgupta2014_knn_density}]
    Let $f$ satisfy Assumption \ref{assumption:differentiability}. Consider any $x\in\Mcal$. Then there exists a neighborhood $B(x,r), r>0$, and constants $\Chat_x, \Ccheck_x>0$ such that, for all $x'\in B(x,r)$, we have
    \begin{equation}\label{eq:mode_quadratic_bound}
        \Ccheck_x\|x'-x\|^2\leq f(x)-f(x')\leq \Chat_x\|x'-x\|^2.
    \end{equation}
\end{lemma}




The above lemma allows us to parameterize any mode $x\in\Mcal$ locally as defined below, which in turn ensures that $A_x\cap\Mcal =\{x\}$.
\begin{definition}\label{def:local_parameterization}
    For every mode $x\in\Mcal$, there exists a $r_x>0$, such that $B(x,r_x)$ is connected in a set $A_x$, satisfying (i) $A_x$ is a connected component\footnote{For any $S\subset\Scal^{d-1}$, we say that $x$ and $y$ are connected is $S$ if $x=\Pcal(0)$ and $y=\Pcal(1)$ where $\Pcal$ is any continuous (path) function $\Pcal: [0,1] \rightarrow S$. We say $S$ is connected if it has a single connected component.} of a level set $\Scal^{d-1}_{\lambda_x} \triangleq \{x'\in\Xcal: f(x')>\lambda_x\}$ for some $\lambda_x>0$, and, (ii) $\exists \Chat_x,\Ccheck_x >0, \forall x'\in A_x, \Ccheck_x\|x'-x\|^2\leq f(x)-f(x')\leq \Chat_x\|x'-x\|^2$. 
\end{definition}
Finally, we assume that every hill in $f$ corresponds to a mode in $\Mcal$. 
\begin{assumption}\label{assumption:level_set_connected_component_mode_containing}
    Each connected component of any level set $\Scal^{d-1}_{\lambda}=\{x\in\Scal^{d-1}: f(x)\geq \lambda\}, \lambda>0$, contains a mode in $\Mcal$.
\end{assumption}

\subsection{Mode recovery guarantee for a unimodal density}
We start with a simple setting when $f$ has a single mode $x_0$ (that is, $|\Mcal|=1$). Our mode estimation algorithm is simple. We estimate the mode of $f$ to be the maximizer of $\fhat_n$ over $X_n$, where $\fhat_n$ is computed using \cref{eqn:eas_density_estimator}.
\begin{equation}\label{eqn:single_mode_estimate}
    \hat{x}_0 = \argmax_{x\in X_n} \fhat_n(x)
\end{equation}
We provide some intuition for why \cref{eqn:single_mode_estimate} is a reasonable strategy. \cref{assumption:differentiability}, specifically \cref{def:local_parameterization}, allows us to easily express an upper and a lower bound of density $f$ at any point $x$ in a small neighborhood of $x_0$ relative to $f(x_0)$. In addition, \cref{lem:relative_density_bound} allows us to approximate $f$ at different scales in different parts of the space by varying $\epsilon$. Judiciously combining these two facts allows us to ensure that the infimum of $\fhat_n(x)$ for any $x$ within a small ball around $x_0$ is strictly larger than the supremum of $\fhat_n(x)$ for any $x$ outside this ball. Our main result of this section for single mode recovery is the following.
\begin{theorem}\label{thm:single_mode_guarantee}
    Let $\delta>0$.  Assume that $f$ has a single mode $x_0$ and satisfies \cref{assumption:differentiability}. Suppose $k$ satisfies
    \[(i) \min\left\{\alpha\cdot C_{\delta}^2d\log m, C_{\delta}^2d\log n\right\}\leq k\]
    \[ii) k\leq \min\{\beta_1\cdot m^{\frac{4}{d+3}},\beta_2\cdot n^{\frac{4}{d+3}}\}\]
    where, $\alpha = \max\left\{4,\left(\frac{f(x_0)}{3}\right)^2,\left(\frac{32f(x_0)}{\Ccheck_{x_0}r_{x_0}^2}\right)^2\right\}, \beta_1 = \frac{1}{6}\left(\left(\frac{3f(x_0)C_{\delta}}{16\Chat_{x_0}}\right)^2\log m\right)^{\frac{d-1}{d+3}}$, and $\beta_2 = \left(\frac{S_{d-1}}{12\sqrt{d}}\right)^{\frac{4}{d+3}}\left(f(x_0)\right)^{\frac{2d+2}{d+3}}\left(\left(\frac{3C_{\delta}}{4\Chat_{x_0}}\right)^2d\log m\right)^{\frac{d-1}{d+3}}$. Let $\hat{x}_0$ be our mode estimate using \cref{eqn:single_mode_estimate}. If $n \geq \frac{9m\log(m/\delta)\max\{\|f\|_{\infty},1\}}{S_{d-1}(f(x_0)C_{\delta})^2d\log m}$, then with probability at least $1-2\delta$, we have,
    \[\|\hat{x}_0-x_0\|\leq 6\sqrt{\frac{C_{\delta}}{\Ccheck_{x_0}}f(x_0)}\left(\frac{d\log m}{k}\right)^{1/4}.\]
\end{theorem}

\begin{remark}
In \cref{thm:single_mode_guarantee}, the upper-bound on $k$ is due to two different events: the first term roughly translates to $k =\tilde{O}\left(m^{\frac{4}{d+3}}\right)$ and ensures that all $\diam(C_j)$ are sufficiently small so that \cref{lem:relative_density_bound} can be applied, while the second term, which roughly translates to $k = \tilde{O}\left(n^{\frac{4}{d+3}}\right)$, ensures that there are enough data points from the sample $X_n$ that fall within a small ball centered at the mode $x_0$, 
\end{remark}

\begin{proof}[Proof sketch]
We define, $r_n(x_0) = \inf\{r: B(x_0,r)\cap X_n\neq \emptyset\}$
and choose $\tau\in(0,1)$ and $\rtilde$  satisfying $2r_n(x_0)/\tau \leq \rtilde \leq r_{x_0}$.
Since $\rtilde\geq 2r_n(x_0)$, the crux of our proof technique relies on establishing, 
\begin{equation}\label{eqn:single_mode_proof_crux}
\sup_{x\in\Scal^{d-1}\setminus B(x_0,\rtilde)} \fhat_n(x) < \inf_{x\in B(x_0,r_n(x_0))} \fhat_n(x)
\end{equation}
as this will ensure that $\|\xhat_0-x\|\leq\rtilde$. To see this, note that if $\sup_{x\in B(x_0,\rtilde)\setminus B(x_0,r_n(x_0))} \fhat_n(x) > \sup_{x\in \Scal^{d-1}\setminus B(x_0,\rtilde)} \fhat_n(x)$, then it must be the case that $\xhat_0\in B(x_0,\rtilde)$. On the other hand, if $\sup_{x\in B(x_0,\rtilde)\setminus B(x_0,r_n(x_0))} \fhat_n(x) \leq \sup_{x\in \Scal^{d-1}\setminus B(x_0,\rtilde)} \fhat_n(x)$, then $\xhat_0\in B(x_0, r_n(x_0))\subset B(x_0, \rtilde)$. When $\rtilde$ is appropriately chosen, this yields the claimed result. 

To prove \cref{eqn:single_mode_proof_crux}, we repeatedly invoke \cref{lem:relative_density_bound} with different $\epsilon$ values. In particular, we find the upper bound of the l.h.s. of \cref{eqn:single_mode_proof_crux} by invoking \cref{lem:relative_density_bound} with properly tuned $\epsilon$, and show that this is less than the lower bound of the r.h.s. of \cref{eqn:single_mode_proof_crux} obtained by invoking again \cref{lem:relative_density_bound} with properly tuned (but different than before) $\epsilon$.
\end{proof}

Now, with the appropriate choice of hyperparameters, we immediately get the following corollary.
\begin{corollary}\label{cor:single_mode_guarantee}
    Setting $m = \Theta(n)$, and $k=O\left(\left(d\log m\right)^{\frac{d-1}{d+3}}n^{\frac{4}{d+3}}\right)$ in \cref{thm:single_mode_guarantee}, we have that with probability at least $1-2\delta$,
    \[\|\hat{x}_0-x_0\| = O\left(\left(\frac{d\log m}{n}\right)^{\frac{1}{d+3}}\right) = \tilde{O}\left(n^{-\frac{1}{d+3}}\right).\]
\end{corollary}
When $f$ is restricted to $\Scal^{d-1}$, there are effectively $d-1$ degrees of freedom, and we note that the rate of \cref{cor:single_mode_guarantee}, ignoring logarithmic factors, is optimal for any twice-differentiable density $f$ \citep{tsybakov1990_mode, donoho1991_rates_of_convergence} and matches the results of \cite{dasgupta2014_knn_density}.


  
  
    
    
    
  


\subsection{Mode recovery guarantee in presence of multiple modes}
Next, we address the mode estimation problem when $f$ has an unknown number of modes. Our solution technique relies on estimating the connected components (CCs) of various level sets of the unknown density, so that, when the modes are sufficiently well separated from each other by apropriate low density regions, the maximizer of $\fhat_n$ within any CC of an appropriate level set will be a good approximation of the mode contained within that CC. It is well known that CCs of the mutual $k$-NN graph and its variants constructed using samples drawn \iid ~ from $f$  are good estimates of the CCs of the corresponding levels sets of the unknown density $f$ \citep{kpotufe2011_pruning,chaudhuri2010cluster_tree,balakrishnan2013_clustr_tree_manifold,dasgupta2014_knn_density}. 
In order to estimate the CCs of level sets of $f$, we define the set of nested graphs $G(\lambda)$ using $\fhat_n$ as given in \cref{eqn:eas_density_estimator}, which are sub-graphs of a mutual $k$-NN graph on the sample $X_n$.
\begin{definition}\label{def:density_graph}
    Given $\lambda\in \RR$, let $G(\lambda)$ denotes the graph whose vertices are in $X_{n,\lambda} \triangleq \{x\in X_n: \fhat_n(x)\geq\lambda\}$, and whose vertices $x, x'$ are connected by an edge if $\|x-x'\|\leq \alpha\cdot\min\{r_k(x),r_k(x')\}$, for some $\alpha\geq\sqrt{2}$, where $r_k(y)$ is the distance from $y$ to its $k^{th}$ nearest neighbor in $X_n$.
\end{definition}

When the density $f$ has multiple modes, we guarantee that the modes which are well separated by other modes by  appropriate low density regions -- termed as \emph{salient} modes -- can be recovered under mild conditions. In order to define salient modes we first define the notion of $r$-separation as in \cite{chaudhuri2010cluster_tree}.
\begin{definition}[$r$-separation]\label{def:r_separation}
$A, A'\subset\Scal^{d-1}$ are $r$-separated if there exists a separating set $S\subset \Scal^{d-1}$ such that every path from $A$ to $A'$ crosses $S$ and $\sup_{x\in S_r}f(x) < \inf_{x\in A\cup A'}f(x)$.   
\end{definition}
Following \cite{dasgupta2014_knn_density}, we define a mode $x$ to be $r$-salient if the critical set $A_x$ of \cref{def:local_parameterization} is well separated from all components at the level where it appears.

\begin{definition}[$r$-salient modes]
    A mode $x$ of $f$ is said to be $r$-salient for $r>0$ if the following holds. There exists $A_x$ as in \cref{def:local_parameterization}, which is a CC of say $\Scal^{d-1}_{\lambda_x} \triangleq \{x\in\Scal^{d-1} : f(x)\geq \lambda_x\}$ and $A_x$ is $r$-separated from $\Scal^{d-1}_{\lambda_x}\setminus A_x$.
\end{definition}

\begin{algorithm}[tb]{\small
   \caption{Recovery of any mode}
   \label{alg:multiple_mode_estimation}
\begin{algorithmic}
   \STATE {\bfseries Input:} $X_n, \tilde{\epsilon}$
   \STATE Initialize $\Mcal_n\leftarrow \emptyset$.
   \FOR{$i=1$ { to} $n$, { in decreasing order of} $\fhat_n(x_i)$}
   \STATE $\lambda \leftarrow \fhat_n(x_i)$
   \STATE Let $\{\tilde{A}_j\}_{j=1}^{m_i}$ be the CC of $G(\lambda-\tilde{\epsilon})$ disjoint from $\Mcal_n$
   \STATE $\Mcal_n \leftarrow \Mcal_n \cup \left\{x'_j \triangleq\argmax_{x\in\tilde{A}_j\cap X_{n,\lambda}} \fhat_n(x) \right\}_{j=1}^{m_i}$ 
   \ENDFOR
   \STATE {\bfseries Return } $\Mcal_n$
\end{algorithmic}}
\end{algorithm}
Our proposed approach is summarized in Alg. \ref{alg:multiple_mode_estimation}, where, $\tilde{\epsilon} = \tilde{\epsilon}(n)$ is an user defined hyper-parameter which is a decreasing function of $n$. In particular, choosing $\tilde{\epsilon}=\gamma_n$, where $\gamma_n$ was defined in Lemma \ref{lem:relative_density_bound} suffices. Alg. \ref{alg:multiple_mode_estimation} outputs a set of estimated modes $\Mcal_n$. Our next theorem guarantees that any salient mode in $\Mcal$ can be recovered by Alg. \ref{alg:multiple_mode_estimation} at the optimal rate $\tilde{O}\left(n^{-\frac{1}{d+3}}\right)$ for the choice of $k=\tilde{O}\left(n^{\frac{4}{d+3}}\right)$ and $m = \Theta(n)$.

\begin{theorem}\label{thm: multiple_modes_guarantee}
    Assume $f$ satisfies Assumptions \ref{assumption:differentiability} and \ref{assumption:level_set_connected_component_mode_containing}. Suppose $k$ satisfies
    $(i) \min\left\{\alpha_1\cdot C_{\delta}^2d\log m, C_{\delta}^2d\log n\right\}\leq k$, and (ii) $k\leq \min\{\beta_1\cdot m^{\frac{4}{d+3}},\beta_2\cdot n^{\frac{4}{d+3}}\}$, where, \\
    $\alpha_1 = \max\left\{4,\left(\frac{f(x_0)}{3}\right)^2,\left(\frac{40f(x_0)}{\Ccheck_{x_0}\min\left\{\left(\frac{r_{x_0}}{2}\right)^2, \left(\frac{r}{\alpha}\right)^2\right\}}\right)^2\right\}$,\\
    $\beta_1 = \frac{1}{6}\left(\left(\frac{15\lambda_{x_0}C_{\delta}}{64\Chat_{x_0}}\right)^2\log m\right)^{\frac{d-1}{d+3}}$, and 
    $\beta_2 = \left(\frac{S_{d-1}}{6\sqrt{d}}\right)^{\frac{4}{d+3}}\left(\lambda_{x_0}\right)^{\frac{2d+2}{d+3}}\left(\left(\frac{15C_{\delta}}{16\Chat_{x_0}}\right)^2d\log m\right)^{\frac{d-1}{d+3}}$. Let $x_0$ be an $r$-salient mode for some $r>0$ and $\Mcal_n$ be the set of modes returned by Alg. \ref{alg:multiple_mode_estimation}. If $n \geq \frac{9m\log(m/\delta)\max\{\|f\|_{\infty},1\}}{S_{d-1}(f(x_0)C_{\delta})^2d\log m}$, then with probability at least $1-2\delta$, there exists $\hat{x}_0\in\Mcal_n$ such that,
    \[\|\hat{x}_0-x_0\|\leq 7\sqrt{\frac{C_{\delta}}{\Ccheck_{x_0}}f(x_0)}\left(\frac{d\log m}{k}\right)^{1/4}.\]
\end{theorem}

\begin{figure*}[ht]
     \centering
     \begin{subfigure}[b]{0.24\textwidth}
         \centering
         \includegraphics[width=\textwidth]{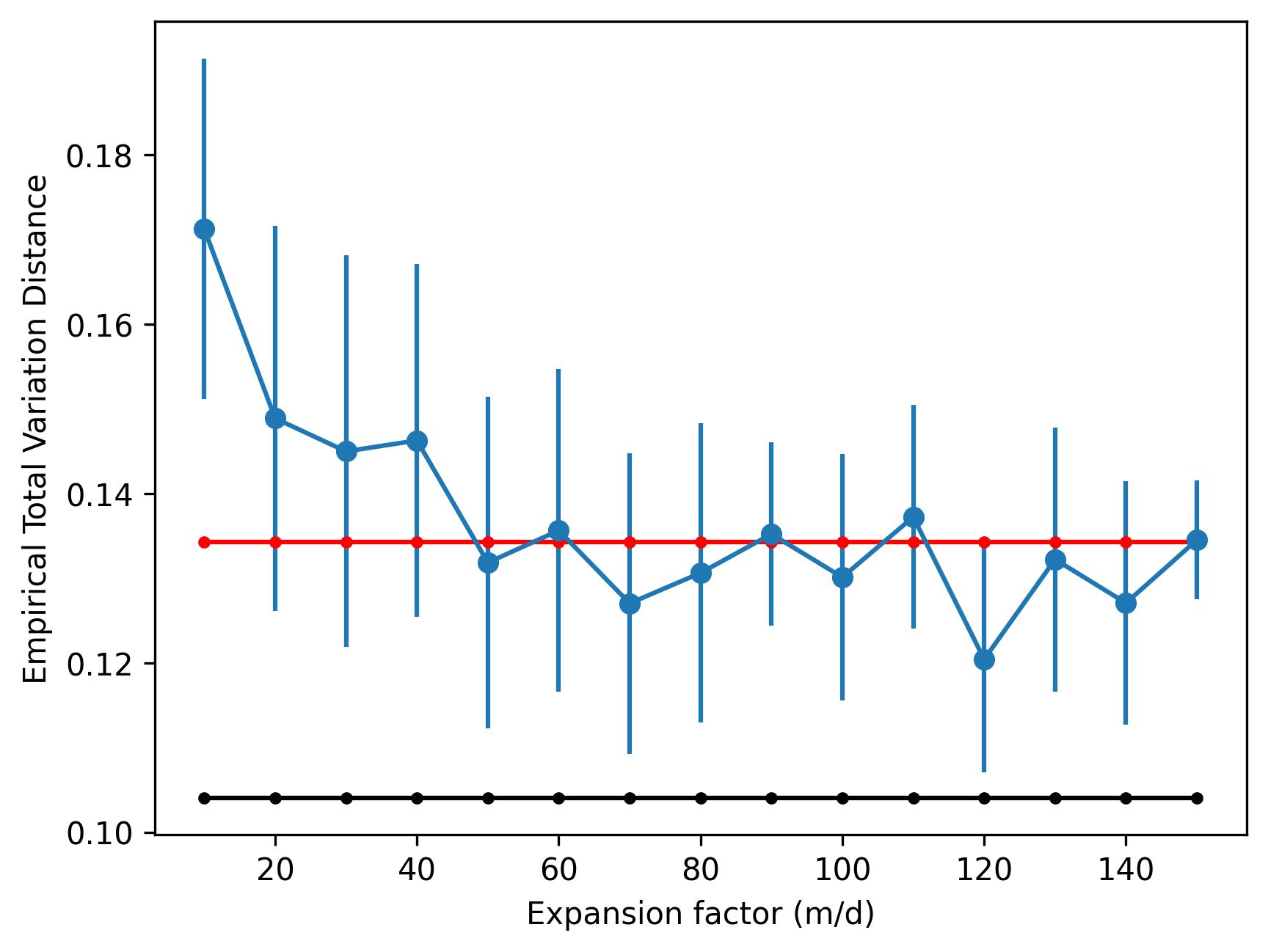}
         \caption{$d=2, \kappa_1=10, \kappa_2=5$}
         \label{fig:y equals x}
     \end{subfigure}
     \hfill
     \begin{subfigure}[b]{0.24\textwidth}
         \centering
         \includegraphics[width=\textwidth]{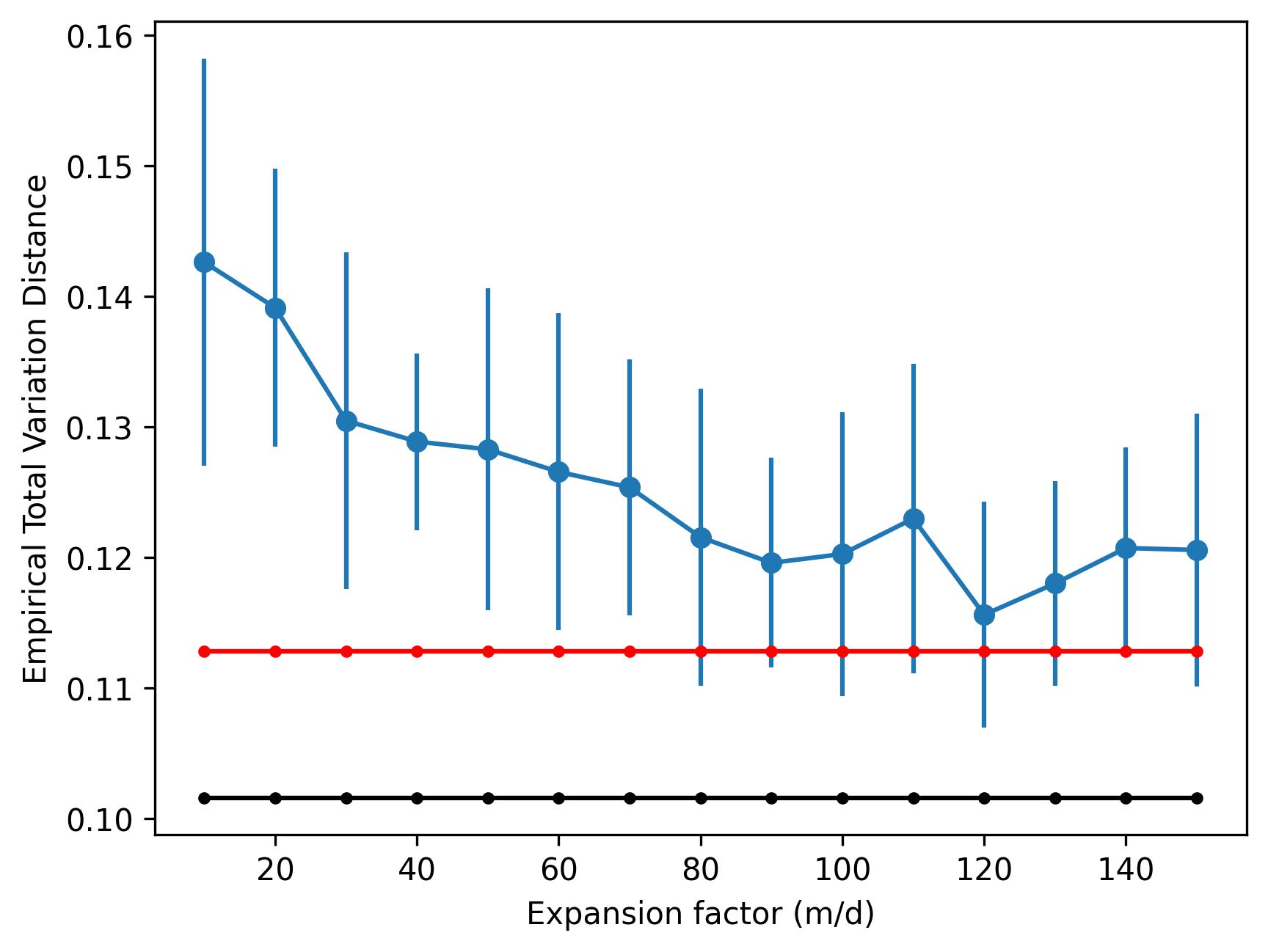}
         \caption{$d=3, \kappa_1=10, \kappa_2=5$}
         \label{fig:three sin x}
     \end{subfigure}
     \hfill
     \begin{subfigure}[b]{0.24\textwidth}
         \centering
         \includegraphics[width=\textwidth]{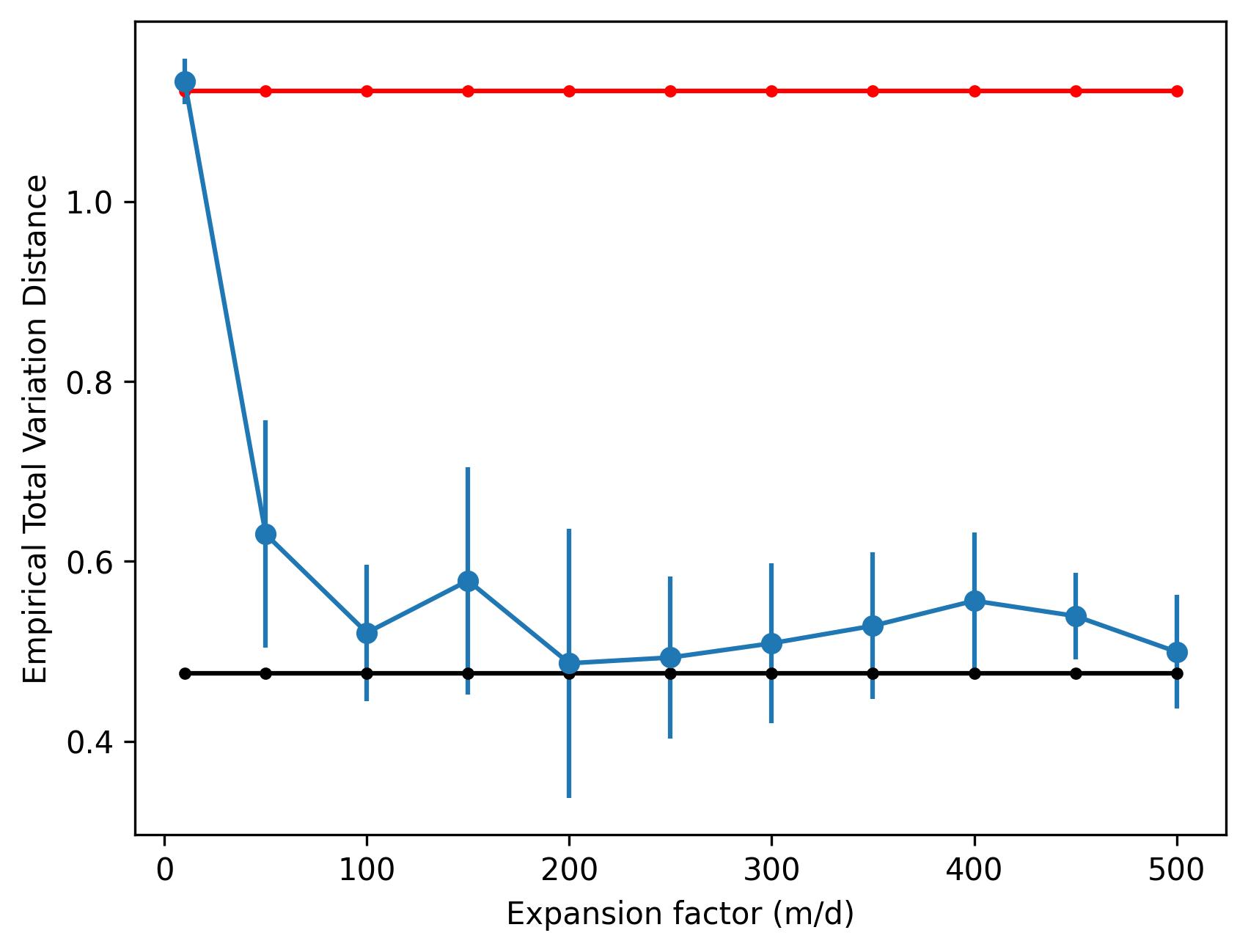}
         \caption{$d=2, \kappa_1=80, \kappa_2=100$}
         \label{fig:five over x}
     \end{subfigure}
     \hfill
     \begin{subfigure}[b]{0.24\textwidth}
     \centering
     \includegraphics[width=\textwidth]{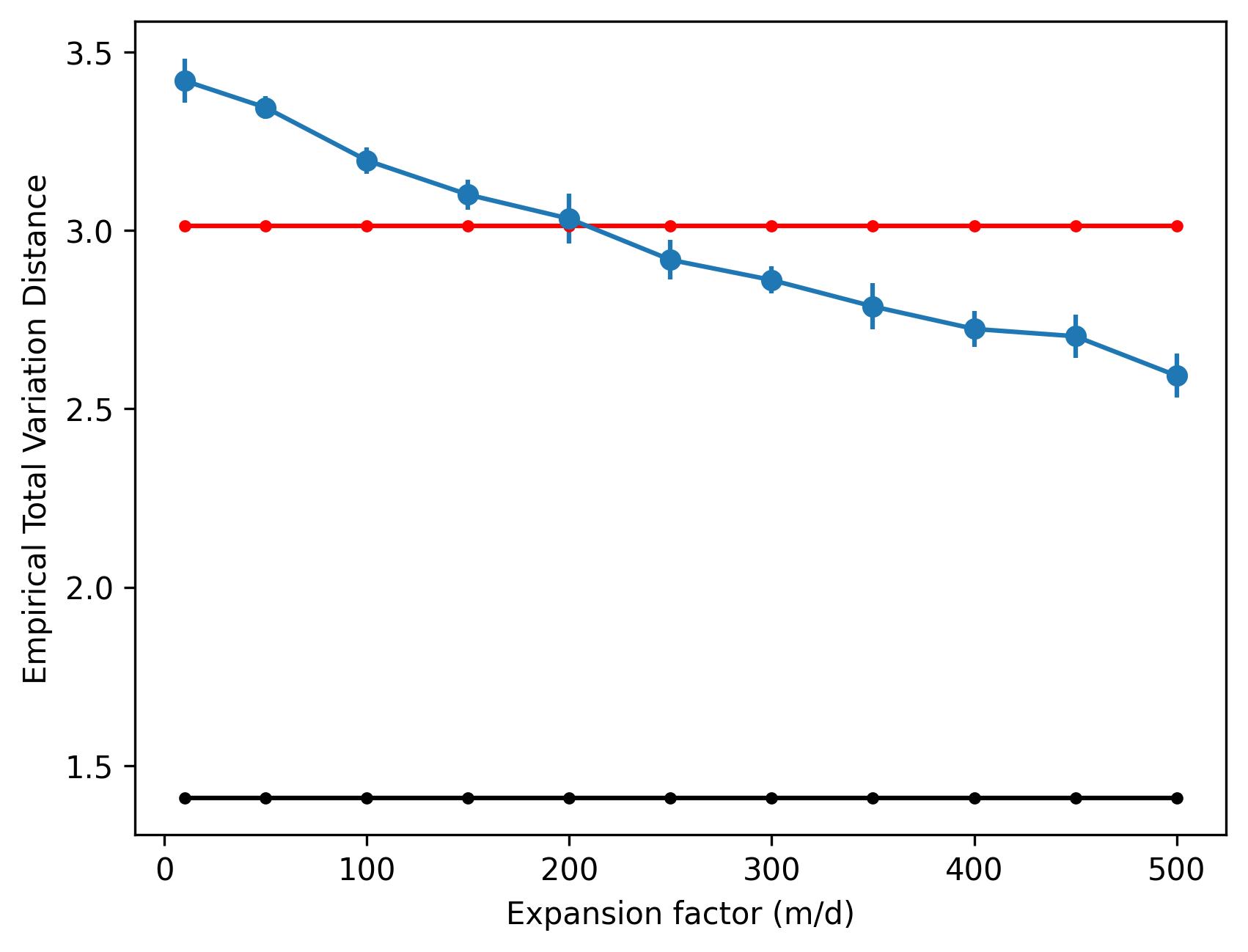}
     \caption{$d=3, \kappa_1=80, \kappa_2=100$}
     \label{fig:five over x}
     \end{subfigure}
        \caption{Color code: Red -- KNNDE, Black -- KDE,  BLUE (error bar) -- EaSDE.}
        \label{fig:density_estimation_results}
\end{figure*}
\section{Empirical evaluations}\label{sec:empirical}
Since our results hold for data lying on the unit sphere, for our empirical evaluations, we choose distributions that are supported on the unit sphere, namely, von Mises-Fisher distributions. 
For any $x\in\mathcal{S}^{d-1}$, the probability density of a von Mises-Fisher distribution is:
\begin{equation}\label{eq:von_mises_fisher}
    f(x) = \frac{\kappa^{\frac{d}{2}-1}}{(2\pi)^{\frac{d}{2}}I_{\frac{d}{2}-1}(\kappa)}\exp\left(\kappa\mu^{\top}x\right)
\end{equation}
where, $\mu\in\mathcal{S}^{d-1}$ is the mean direction, $\kappa>0$ is the concentration parameter, $d$ is the dimension, and $I$ is the modified Bessel function of the first kind.  The distribution becomes narrower around $\mu$ with increasing $\kappa$. 
For our empirical evaluations, we sample data points from a mixture of two von Mises-Fisher distributions given by:
\begin{equation}\label{eq:von_mises_fisher_mixture}
    f(x) = wf_1(x)+(1-w)f_2(x)
\end{equation}
where, for $i=1,2$,  $f_i(x)$ is specified,  using \eqref{eq:von_mises_fisher}, by mean direction $\mu_i$, concentration parameter $\kappa_i$ and $w\in(0,1)$ is the mixture coefficient. We consider data dimensionality $d=2$ and $3$ for our experiments.
We compare our proposed density estimation method (EaSDE) with kNN based density estimation method (KNNDE) of \cite{dasgupta2014_knn_density} and kernel density estimation method (KDE). For each of these methods we train the respective density estimators using the training set, select  hyper-parameters using the validation set\footnote{See Appendix \ref{sec:empirical_evaluations_extra} for details of hyper-parameter selection, data generation process and other experimental details.} and evaluate the performance on the test set. We measure the accuracy of the density estimators using empirical total variation distance (ETV) defined by $ETV(f,\hat{f})=\frac{1}{2M}\sum_{i=1}^M|f(x_i)-\hat{f}(x_i)|$, where, $f$ is the true density, $\hat{f}$ is the estimated density and $x_1,\ldots,x_M$ is the test set. We present our results in Figure \ref{fig:density_estimation_results}, where we plot expansion factor, which is $m/d$, on the $x$-axis and plot the ETV on the $y$-axis. As can be seen from Figure \ref{fig:density_estimation_results}, ETV for EaSDE generally decreases with increasing expansion factor. In sub-figure (a) and (b), we use smaller $\kappa_1$ and $\kappa_2$ values resulting in a flatter true distribution and using a reasonably smaller expansion factor, EaSDE achieves comparable performance to that of KNNDE and KDE. In sub-figure (c) and (d), we used higher values for $\kappa_1$ and $\kappa_2$ resulting in a concentrated component distributions around the component means. In this case, ETV is much higher for all methods, and KNNDE in particular performs poorly. Here EaSDE requires a larger expansion factor to reduce ETV, more so for higher $d$.

\section{Conclusions and limitations}\label{sec:conclusions}
In this paper we studied the suitability of expand-and-sparsify representations for two fundamental statistical problems -- density estimation and mode estimation. For density estimation, we demonstrated that given $n$ \iid ~examples from a density $f$, there is an estimator $\fhat_n$ that is a linear function of the expand-and sparsify representation that achieves a minimax-optimal convergence rate 
in $\ell_{\infty}$ distance.
Our bounds hold with high probability, simultaneously for all $x$ in the support of $f$. For mode estimation, we presented simple algorithms on top of this density estimator and demonstrated that under mild conditions, modes of $f$ can be recovered at a rate $\tilde{O}\left(n^{-\frac{1}{d+3}}\right)$ with high probability.

For mode estimation, we assume that $f$ is twice differentiable in a neighborhood of each mode and our mode recovery rate is optimal up to a logarithmic factor. It is not immediately clear whether this extra logarithmic factor is an artifact of our proof or it is a limitation of our algorithm. 

If the data is on a manifold that is not supported on the sphere, then our rates do not hold, since the sparsification scheme implies that $h(cx) = h(x)$ for any $c > 0$. But if the data comes from a submanifold of the sphere, then our rates still hold, but they are not adaptive to the intrinsic dimension of the submanifold. It is an interesting future direction to consider expand-and-sparsify schemes that adapt to the intrinsic dimension of a dataset.

The expand-and-sparsify representation also bears a strong resemblance to certain algorithms arising in the hyperdimensional computing literature. As a future research direction, it would be very interesting to explore the computational benefits of the expand-and-sparsify representation on non-traditional computer architectures.



\subsubsection*{Acknowledgements} We thank the anonymous reviewers for their constructive feedback. KS gratefully acknowledges funding from the “NSF AI Institute for Foundations of Machine Learning (IFML)” (FAIN:2019844).

\bibliography{refs}

\clearpage
\appendix
\thispagestyle{empty}

\onecolumn
\aistatstitle{Optimal rates for density and mode estimation with expand-and-sparsify representations: 
Supplementary Materials}

\section{Proofs from section \ref{sec:density_rates}}\label{sec:density_rates_proofs}
\subsection{Proof of \cref{lem:inner-outer-balls}}

For a set $B \subset \Scal^{d-1}$, let 
\[\nu_m(B) = \frac{1}{m} \sum_{j=1}^m \ind[\theta_j \in B],\]
where $\theta_1, \ldots, \theta_m$ are the random vectors drawn from $\nu$.

\begin{lemma}[\cite{chaudhuri2010cluster_tree}]
\label{lem:CD-vc-ball-bound}
    Assume $k \geq d \log m$ and fix $\delta > 0$. Then there is an absolute constant $c_o$ such that with probability at least $1-\delta$, the following holds for every ball $B \subset \RR^d$:
    \begin{align*}
        \nu(B) \geq \frac{k}{m}\left(1 + C_\delta \sqrt{\frac{d \log m}{k}} \right) \ & \Longrightarrow \ \nu_m(B) \geq \frac{k}{m} \\ 
        \nu(B) \leq \frac{k}{m}\left(1 - C_\delta \sqrt{\frac{d \log m}{k}} \right) \ & \Longrightarrow \ \nu_m(B) < \frac{k}{m},
    \end{align*}
    where $C_\delta = 2 c_o \log(2/\delta)$.
\end{lemma}

We can use \Cref{lem:CD-vc-ball-bound} to prove \cref{lem:inner-outer-balls}.

\begin{proof}[Proof of \cref{lem:inner-outer-balls}]
First we observe that because $\nu$ is uniform over $\Scal^{d-1}$, for any value $p \in (0,1)$, there is a unique radius $r \in (0,2)$ such that $\nu(B(x, r)) = p$ for any $x \in \Scal^{d-1}$. Observe that the bounds on $k$ imply that
\[ 0 < \frac{k}{m}\left(1 - C_\delta \sqrt{\frac{d \log m}{k}} \right) \leq \frac{k}{m}\left(1 + C_\delta \sqrt{\frac{d \log m}{k}} \right) < 2. \]

The proof that $C_j \subset B(\theta_j, r_{\out})$ follows directly from Lemma 11 in \cite{dasgupta2020expressivity}, but we include the proof here for completeness. Let us condition on the events of \Cref{lem:CD-vc-ball-bound} occurring. 
    
By definition, every ball of radius $r_{\out}$ centered at some $x \in \Scal^{d-1}$ has $\nu$-mass 
\[\nu(B(x, r_{\out})) = \frac{k}{m} \left( 1 + C_\delta \sqrt{\frac{d \log m}{k}}\right).\]
By \Cref{lem:CD-vc-ball-bound}, this implies that there are at least $k$ $\theta_j$'s in each $B(x, r_{\out})$.
Another way to say this is that every $x \in \Scal^{d-1}$ will have its nearest $k$ $\theta_j$'s within distance $r_{\out}$. Thus, any $h_j(x)$ will only be activated for points $x$ within distance $r_{\out}$ of $\theta_j$, i.e. $C_j \subset B(\theta_j, r_{\out})$.

To see that $B(\theta_j, r_{\inner}) \subset C_j$ observe again that by definition, every ball of radius $r_{\inner}$ centered at some $x \in \Scal^{d-1}$ has $\nu$-mass 
\[\nu(B(x, r_{\inner})) = \frac{k}{m} \left( 1 - C_\delta \sqrt{\frac{d \log m}{k}}\right).\]
Thus, by \Cref{lem:CD-vc-ball-bound}, every ball $B(x,r_{\inner})$ will contain fewer than $k$ points. Thus, if $x \in B(\theta_j, r_{\inner})$ then $\theta_j \in B(x, r_{\inner})$, and so $\theta_j$ must be among the $k$ closest $\theta$'s to $x$, and so satisfy that $x \in C_j$.

To finish off the proof, we observe that $\nu(C) = \frac{\vol(C)}{S_{d-1}}$ for any $C \subset \Scal^{d-1}$.
\end{proof}

\subsection{Proof of \cref{lem:empirical_mass_C_j}}

We will use the following form of Bernstein's bound.
\begin{lemma}[Bernstein bound]\label{lem:bernstein_bound}
Let $X_1,\ldots,X_n$ be i.i.d. and suppose that $|X_i|\leq c$ and $\EE (X_i)=\mu$. Pick $0<\delta<1$. Then with probability at least $1-\delta$,
\[\ \left|\frac{1}{n}\sum_{i=1}^nX_i-\mu \right|\leq \sqrt{\frac{2\sigma^2\log(1/\delta)}{n}}+\frac{2c\log(1/\delta)}{3n}\]
where, $\sigma^2=\frac{1}{n}\sum_{i=1}^n\Var(X_i)$.
\end{lemma}

\begin{proof}[Proof of \cref{lem:empirical_mass_C_j}]
Fix any $j\in [m]$ and let  $z_i=\ind[x_i\in C_j]$. Then $|z_i|\leq 1$ and
\begin{align*}
\Var(z_i) &= \EE(z_i^2)-(\EE z_i)^2=f(C_j)-(f(C_j))^2 \\ 
& =f(C_j)(1-f(C_j)\leq f(C_j).
\end{align*}
Combining \cref{lem:bernstein_bound} with a union bound, we can say that with probability at least $1-\delta$,
\begin{align*}
\left| f_n(C_j) - f(C_j) \right| & \leq \sqrt{\frac{2{f(C_j)\log(m/\delta)}}{n}}+\frac{2\log(m/\delta)}{3n}.
\end{align*}
\end{proof}

\subsection{Proof of \cref{lem:relative_density_bound}}

\begin{proof}
Pick any $x\in \Scal^{d-1}$ and assume that $x\in C_j$ for some $j\in \{1,\ldots,m\}$. Since $\diam(C_j)\leq \rhat(\epsilon,x)$, we have $C_j\subset B(x,\rhat(\epsilon,x))$. This implies $f(C_j)\leq \vol(C_j)(f(x)+\epsilon)$. Next, note that, with probability at least $1-\delta$, \eqref{eqn:volume-approximation} holds. Conditioning on this occurring, and for the choice of $k$ as given in the Lemma statement, we have that for every $j=1,\ldots, m$, 
    \[f(C_j)\leq \|f\|_{\infty}\vol(C_j)\leq\frac{2kS_{d-1}\|f\|_{\infty}}{m}\]
    Using this bound, Lemma \ref{lem:empirical_mass_C_j} ensures that with probability at least $1-\delta$,
    \[|f_n(C_j)-f(C_j)|\leq 2\sqrt{\frac{kS_{d-1}\|f\|_{\infty}\log (m/\delta)}{mn}}+\frac{2\log (m/\delta)}{3n} = \alpha_n\]
    We condition on this occurring.
    We first show the upper bound.

    \begin{align*}
        \fhat_n(x) &= \frac{m}{k^2S_{d-1}}\sum_{j:h_j(x)=1}f_n(C_j) \leq \frac{1}{k}\sum_{j:h_j(x)=1}\left(\frac{f(C_j)}{S_{d-1}k/m} + \frac{\alpha_n}{S_{d-1}k/m}\right) \\
        & \leq \frac{1}{k}\sum_{j:h_j(x)=1}\left(\frac{f(C_j)}{\vol(C_j)}\left(1 + C_{\delta}\sqrt{\frac{d\log m}{k}}\right) + \frac{\alpha_n}{S_{d-1}k/m}\right) \\
        & \leq \frac{1}{k}\sum_{j:h_j(x)=1}\left(\frac{\vol(C_j)(f(x)+\epsilon)}{\vol(C_j)}\left(1 + C_{\delta}\sqrt{\frac{d\log m}{k}}\right) + \frac{\alpha_n}{S_{d-1}k/m}\right) \\
        &= \left(1+C_{\delta}\sqrt{\frac{d\log m}{k}}\right)(f(x)+\epsilon)+\gamma_n
    \end{align*}

The lower bound is shown in a similar manner. 
Pick any $x\in \Scal^{d-1}$ and assume that $x\in C_j$ for some $j\in \{1,\ldots,m\}$. Since $\diam(C_j)\leq \rcheck(\epsilon,x)$, we have $C_j\subset B(x,\rcheck(\epsilon,x))$. 

This implies $f(C_j)\geq \vol(C_j)(f(x)-\epsilon)$. Then we have, 
    
        \begin{align*}
        \fhat_n(x) &= \frac{m}{k^2S_{d-1}}\sum_{j:h_j(x)=1}f_n(C_j) \geq \frac{1}{k}\sum_{j:h_j(x)=1}\left(\frac{f(C_j)}{S_{d-1}k/m} - \frac{\alpha_n}{S_{d-1}k/m}\right) \\
        & \geq \frac{1}{k}\sum_{j:h_j(x)=1}\left(\frac{f(C_j)}{\vol(C_j)}\left(1 - C_{\delta}\sqrt{\frac{d\log m}{k}}\right) - \frac{\alpha_n}{S_{d-1}k/m}\right) \\
        & \geq \frac{1}{k}\sum_{j:h_j(x)=1}\left(\frac{\vol(C_j)(f(x)-\epsilon)}{\vol(C_j)}\left(1 - C_{\delta}\sqrt{\frac{d\log m}{k}}\right) - \frac{\alpha_n}{S_{d-1}k/m}\right) \\
        &= \left(1 - C_{\delta}\sqrt{\frac{d\log m}{k}}\right)(f(x)-\epsilon)-\gamma_n. \qedhere
        \end{align*} 
\end{proof}

\subsection{Proof of \cref{lem:general-diameter-bound}}

We will make use of the following result.

\begin{lemma}[Lemma~12~\cite{dasgupta2020expressivity}]
\label{lem:mass-to-radius}
Suppose $\nu$ is uniform over $\Scal^{d-1}$, $\theta \in \Scal^{d-1}$, and $r \in (0,1)$. Then
\[ r \leq \frac{2}{\sqrt{3}}\left( 3 \sqrt{d} \nu(B(x,r)) \right)^{1/(d-1)}. \]
\end{lemma}

\begin{proof}[Proof of \cref{lem:general-diameter-bound}]
For any $C_j$, observe that by the conclusion of \cref{lem:inner-outer-balls}, we have $C_j \subset B(\theta_j, r_{\out})$, where 
\[ \nu(B(\theta_j, r_{\out})) \leq \frac{k}{m} \left( 1 + C_\delta \sqrt{\frac{d \log m}{k}}\right) \leq \frac{2k}{m}, \]
where we have used our lower bound on $k$. Thus, by \cref{lem:mass-to-radius},
\begin{align*}
\diam(C_j) &\leq 2 r_{\out} \\
&\leq \frac{4}{\sqrt{3}} \left(3\sqrt{d} \nu(B(x, r_{\out})\right)^{1/(d-1)} \\
&\leq \frac{4}{\sqrt{3}} \left(\frac{6 \sqrt{d} k}{m}\right)^{1/(d-1)}. \qedhere
\end{align*}
\end{proof}

\subsection{Proof of \cref{thm:smooth-density-bound}}
\begin{proof}
First observe that with probability at least $1-2\delta$, the conclusions of \cref{lem:inner-outer-balls} and \cref{lem:empirical_mass_C_j} hold. Assume that this is so, in which case \cref{lem:relative_density_bound} and \cref{lem:general-diameter-bound} both hold as well. 

We will select $\epsilon = \sqrt{\frac{1}{k}}$. Now observe that by the condition of smoothness, if we have $\epsilon \leq L r^\beta$, then 
\[ \rhat(\epsilon,x), \rcheck(\epsilon, x) \geq \left( \frac{\epsilon}{L}\right)^{1/\beta} \]
for all $x \in \Scal^{d-1}$.
Our choice of $\epsilon, k$ ensures the above does hold. Applying \cref{lem:general-diameter-bound}, we have
\begin{align*}
\diam(C_j) &\leq \frac{4}{\sqrt{3}} \left( \frac{6 \sqrt{d} k}{m} \right)^{1/(d-1)} \leq L^{-1/\beta} k^{-1/2\beta} = \left(\frac{\epsilon}{L}\right)^{\beta},
\end{align*}
where the second inequality follows from our lower bound on $m$ and canceling terms and the final equality follows from our choice $\epsilon$. Thus, the conditions of \cref{lem:relative_density_bound} are met, and we have that for any $x \in \Scal^{d-1}$,
\begin{align*}
|f(x) - \hat{f}_n(x)| 
&\leq f(x) C_\delta \sqrt{\frac{d \log m}{k}} + \epsilon + \epsilon C_\delta \sqrt{\frac{d \log m}{k}} + \gamma_n \\
&\leq \left( \|f\|_\infty C_\delta \sqrt{d \log m} + 2 \right) k^{-1/2} + \gamma_n,
\end{align*}

where we have used our choice of $\epsilon = \sqrt{1/k}$ and our lower bound of $k \geq C_\delta^2 d \log m$.
\end{proof}


Instantiating \cref{thm:smooth-density-bound} with $m = \Theta(n)$ and $k = \Theta\left( m^{\frac{2\beta}{2\beta+(d-1)}}\left(\log m\right)^{\frac{d-1}{2\beta+(d-1)}}\right)$ leads to $\sup_{x\in\Scal^{d-1}}|f(x)-\fhat_n(x)| = \|f-\fhat_n\|_{\infty} = O\left(\left(\frac{\log n}{n}\right)^{\frac{\beta}{2\beta+(d-1)}}\right)$. 

\subsection{Proof of \cref{thm:density_lower_bound}}
\label{sec:lower_bound}
\begin{proof}
The high level idea of our proof is to first find a packing of the unit sphere, then construct a finite set of smooth functions, using such packing, with non-overlapping support. Finally, construct a finite family of $L,\beta$ smooth densities using the smooth functions constructed before for which we can use Fano's minimax bound (Theorem \ref{th:fanos_inequality}). Now we provide the details. For any positive integer $k$, we use the notation $[k]$ to denote the set $\{1,\ldots,k\}$.

Fix any $a\in \Scal^{d-1}$ and let $k : \Scal^{d-1}\rightarrow \RR$ be a bounded function defined on $B(a,1/2)\cap \Scal^{d-1}$ such that, \[\int_{x\in B(a,1/2)\cap \Scal^{d-1}}k(x)dx = 0, ~\forall x,x'\in\Scal^{d-1}, |k(x)-k(x')|\leq \|x-x'\|.\]
Note that, based on our function definition, $k(x)=0$ for any $x \not\in B(a,1/2)\cap \Scal^{d-1}$. Also, let $k_{\max} = \max_{x\in\Scal^{d-1}}|k(x)|$.

For any $h>0$, using Lemma \ref{lem:packing_number_upper_bound}, let $m =\frac{\sqrt{d}/2}{h^{d-1}}$ be the packing number of $\Scal^{d-1}$. Then there exists $S = \{x_1,\ldots,x_m\}\subset \Scal^{d-1}$ such that, $\|x_i-x_j\|>h$ and $C_i\cap C_j = \emptyset$ for $i\neq j$, where $C_i = B(x_i,h/2)\cap\Scal^{d-1}$ for $i=1,\ldots,m$. For any such $x_j\in S$, define
\[g_j(x) = h^{\beta}k\left(\frac{x-x_j}{h}+a\right).\]

Let $y = \frac{x-x_j}{h}+a$. Then $dy = \left(\frac{1}{h}\right)^ddx$.
It is easy to check that $g_j$ has support in $C_i$ and therefore,
\begin{align*}
    \int_{\Scal^{d-1}}g_j(x)dx &= h^{\beta}\int_{\Scal^{d-1}}k\left(\frac{x-x_j}{h}+a\right)dx\\
    &=h^{\beta+d}\left(\int_{y\in B(a,1/2)\cap\Scal^{d-1}}k(y)dy + \int_{y\notin B(a,1/2)\cap\Scal^{d-1}}k(y)dy\right) =0
\end{align*}

Thus, $g_j$s have non-overlapping support and each $g_j$ is 1-Lipschitz. To see this, observe that for $\forall x,x'\in\Scal^{d-1}$,
\[|g_j(x)-g_j(x')|=h^{\beta}\left| k\left(\frac{x-x_j}{h}+a\right)-k\left(\frac{x'-x_j}{h}+a\right) \right|\leq h^{\beta}\left(\frac{\|x-x'\|}{h}\right)^{\beta} = \|x-x'\|^{\beta}\]
Next, we define the function class, where $S_{d-1} = \vol(\Scal^{d-1})$.
\begin{equation}\label{eqn:finite_density_class}
    \Mcal = \left\{f_j(x) = \frac{1}{S_{d-1}}+\frac{g_j(x)}{2S_{d-1}k_{\max}}: j= 1,\ldots,m \right\}
\end{equation}
For $h<1$, it is easy to see that $f_j(x)>0$ for all $x\in\Scal^{d-1}$ and $j=1,\ldots,m$. Also,
\[\int_{\Scal^{d-1}}f_j(x)dx = \frac{1}{S_{d-1}}\int_{\Scal^{d-1}}dx+\frac{1}{2S_{d-1}k_{\max}}\int_{\Scal^{d-1}}g_j(x)dx=1\]
Therefore, each $f_i\in\Mcal$ is a probability density function whose support is $\Scal^{d-1}$. It is easy to verify that that each $f_i\in \mathcal{M}$ is $\frac{1}{S_{d-1}k_{\max}},\beta$ smooth. For each $f_i\in\mathcal{M}$ we denote by $P_i$ its associated probability distribution.

In order to apply Fano's inequality, we need to compute, (i) an upper bound of $\max_{i\neq j} \kl(P_i,P_j)$, and (ii) a lower bound of $\min_{i\neq j}\|f_i-f_j\|_{\infty}$. Let's start with the \kl divergence first. 
Note that for any $x\in C_i$ and $i\in [m]$,
\begin{align}\label{eqn:densities_on_C_i}
    f_i(x) = \frac{1}{S_{d-1}}+\frac{g_i(x)}{2S_{d-1}k_{\max}}, f_j(x) = \frac{1}{S_{d-1}}
\end{align}
Note that for $h\leq1$, for any $i\in [m]$, $f_i(x)\geq \frac{1}{S_{d-1}}-\frac{k_{\max}}{2S_{d-1}k_{\max}}=\frac{1}{2S_{d-1}}$. Using this lower bound, we have,
\begin{align}\label{eqn:KL_bound}
    \kl(P_i,P_j) & \leq \chi^2(f_i,f_j)=\int_{\Scal^{d-1}}\left(\frac{f_i(x)}{f_j(x)}-1\right)^2f_i(x)dx = \int_{S^{d-1}}\frac{(f_i(x)-f_j(x))^2}{f_j(x)}dx \nonumber \\
    & \leq 2S_{d-1}\int_{S^{d-1}}(f_i(x)-f_j(x))^2dx \nonumber \\
    & = 2S_{d-1}\left(\int_{C_i\cup C_j}(f_i(x)-f_j(x))^2dx + \int_{\Scal^{d-1}\setminus\{C_i\cup C_j\}}(f_i(x)-f_j(x))^2dx\right) \nonumber \\
    & \stackrel{a}{=} 2S_{d-1}\int_{C_i\cup C_i}(f_i(x)-f_j(x))^2dx \nonumber \\
    &= 2S_{d-1}\left(\int_{C_i}(f_i(x)-f_j(x))^2dx + \int_{C_j}(f_i(x)-f_j(x))^2dx\right) \nonumber \\
    &\stackrel{b}{\leq} 2S_{d-1}\left(\int_{C_i}\left(\frac{h^{\beta}k_{\max}}{2S_{d-1}k_{\max}}\right)^2dx + \int_{C_j}\left(\frac{h^{\beta}k_{\max}}{2S_{d-1}k_{\max}}\right)^2dx\right)\nonumber \\
    & = \left(\frac{h^{2\beta}}{2S_{d-1}}\right)\left(\int_{C_i}dx + \int_{C_j}dx\right)= \left(\frac{h^{2\beta}}{2S_{d-1}}\right)\left(\vol(C_i)+\vol(C_j)\right) \nonumber \\
    &\stackrel{c}{\leq} \left(\frac{h^{2\beta}}{S_{d-1}}\right)\frac{2S_{d-1}}{\sqrt{d}}h^{d-1} \nonumber \\
    &= \left(\frac{2}{\sqrt{d}}\right)h^{2\beta+(d-1)}.
\end{align}
where, equality $a$ follows from the fact that outside any $C_i\cup C_j$, $f_i$ and $f_j$ are identical. Inequality $b$ follows from \eqref{eqn:densities_on_C_i} and using the maximum possible value of $g_i$. Inequality $c$ follows from Lemma \ref{lem:spherical_cap_volume_upper_bound}.

Now, for $i\neq j$, we have, 
\begin{align}\label{eqn:ell_infinity_lower_bound}
    \|f_i-f_j\|_{\infty} & = \sup_{x\in \Scal^{d-1}} |f_i(x)-f_j(x)| = \sup_{x\in C_i\cup C_j}|f_i(x)-f_j(x)| \nonumber \\
    & \stackrel{a}{=} \sup_{x\in C_i\cup C_j}\left | \frac{g(x)}{2S_{d-1}k_{\max}}\right | =\frac{h^{\beta}}{2S_{d-1}}
\end{align}
where, equality $a$ follows from \eqref{eqn:densities_on_C_i}. 

Now set $h = \kappa\left(\frac{\log n}{n}\right)^{\frac{1}{2\beta+(d-1)}}$ for some $\kappa \leq \min\left\{\frac{1}{16},\left(\frac{\sqrt{d}}{16\left(1+\frac{2\beta}{d-1}\right)}\right)^{\frac{1}{2\beta+(d-1)}}\right\}$. Then it follows from Lemma \ref{lem:kappa_bound} that $\frac{n\max_{i\neq j}\kl(P_i,P_j)+\log 2}{\log m}\leq \frac{1}{2}$. The desired lower bound immediately follows from an application of Theorem \ref{th:fanos_inequality}.
\end{proof}

\begin{lemma}\label{lem:kappa_bound}
    Suppose $d>4$, $n\geq 16$,  and $\kappa \leq \min\left\{\frac{1}{16},\left(\frac{\sqrt{d}}{16\left(1+\frac{2\beta}{d-1}\right)}\right)^{\frac{1}{2\beta+(d-1)}}\right\}$. Set $h = \kappa\left(\frac{\log n}{n}\right)^{\frac{1}{2\beta+(d-1)}}$ and $m = \frac{\sqrt{d}/2}{h^{d-1}}$. Then the following holds.
    \[\frac{n\max_{i\neq j}\kl(P_i,P_j)+\log 2}{\log m}\leq \frac{1}{2}.\]
\end{lemma}
\begin{proof}
    Since $h = \kappa\left(\frac{\log n}{n}\right)^{\frac{1}{2\beta+(d-1)}}$, we have,
    \[\max_{i\neq j}\kl(P_i,P_j)  \leq \frac{2}{\sqrt{d}}h^{2\beta+(d-1)} = \frac{2\kappa^{2\beta+(d-1)}}{\sqrt{d}}\left(\frac{\log n}{n}\right).\]
    Also we have,
    \begin{align*}
        \log m &= \log\left(\frac{\sqrt{d}/2}{h^{d-1}}\right) = \log\left(\left(\frac{\sqrt{d}}{2}\right)\left(\frac{1}{\kappa}\right)^{d-1}\left(\frac{n}{\log n}\right)^{\frac{d-1}{2\beta+(d-1)}}\right) \\
        &= \frac{1}{2}\log\left(\frac{d}{4}\right)+(d-1)\log\left(\frac{1}{\kappa}\right)+ \frac{(d-1)}{2\beta+(d-1)}\left(\log n - \log\log n\right)
    \end{align*}

    Therefore,
    \begin{align*}
        \frac{n\max_{i\neq j}\kl(P_i,P_j)+\log 2}{\log m} &\leq \frac{\frac{2}{\sqrt{d}}\kappa^{2\beta+(d-1)}\log n +\log 2}{\frac{1}{2}\log\left(\frac{d}{4}\right)+(d-1)\log\left(\frac{1}{\kappa}\right)+ \frac{(d-1)}{2\beta+(d-1)}\left(\log n - \log\log n\right)} \\
        &\leq \frac{\frac{2}{\sqrt{d}}\kappa^{2\beta+(d-1)}\log n}{\frac{(d-1)}{2\beta+(d-1)}\left(\log n - \log\log n\right)} + \frac{\log 2}{(d-1)\log\left(\frac{1}{\kappa}\right)} \\
        & =  \frac{2}{\sqrt{d}}\left(1+\frac{2\beta}{d-1}\right)\kappa^{2\beta+(d-1)} \frac{1}{\left(1-\frac{\log\log n}{\log n}\right)}+ \frac{\log 2}{(d-1)\log\left(\frac{1}{\kappa}\right)} \\
        & \leq  \frac{4}{\sqrt{d}}\left(1+\frac{2\beta}{d-1}\right)\kappa^{2\beta+(d-1)}  + \frac{\log 2}{(d-1)\log\left(\frac{1}{\kappa}\right)} \\
        &\leq \frac{1}{4} + \frac{\log 2}{(d-1)\log\left(\frac{1}{\kappa}\right)} \\
        & \leq \frac{1}{4} + \frac{\log 2}{\log\left(\frac{1}{\kappa}\right)} \leq \frac{1}{4} = \frac{1}{4} = \frac{1}{2}.
    \end{align*}
\end{proof}

\begin{theorem}[Fano's minimax bound \citep{tsybakov:2008:nonparametric}]
\label{th:fanos_inequality}
    Let $\mathcal{P}$ be a set of distributions and let $x_1,\ldots,x_n$ be a sample from some distribution $P\in\mathcal{P}$. Let $\hat{\theta}=\hat{\theta}(x_1,\ldots,x_n)$ is an estimator of the parameter $\theta(P)$ taking values in a metric space with metric $d$. Let $F = \{P_1,\ldots,P_N\}\subset \mathcal{P}$. Then,
    \[\inf_{\hat{\theta}}\sup_{P\in\mathcal{P}}\E_P\left(d\left(\hat{\theta}-\theta(P)\right)\right)\geq \frac{s}{2}\left(1-\frac{n\alpha+\log 2}{\log N}\right)\]
where $s= \min_{i\neq j} d\left(\theta(P_i), \theta(P_j)\right)$, and $\alpha = \max_{i\neq j}\kl(P_i,P_j)$.
\end{theorem}

\begin{lemma}\label{lem:packing_number_upper_bound}
    Pick any $h\in (0,1)$. Let $N$ be the largest set of $h$-separated points on $\Scal^{d-1}$ using $\ell_2$ norm, that is, there exists $x_1,\ldots,x_N\in\Scal^{d-1}$ such that for any $i\neq j, ~ \|x_i-x_j\|>h$. Then, $N\leq \frac{(\sqrt{d}/2)}{h^{d-1}}$.
\end{lemma}
\begin{proof}
It is clear that, for $i=1,\ldots, N$, $B(x_i,h/2)$ are disjoint. Then, it must be the case that, $\sum_{i=1}^N\vol(B(x_i,h/2))\leq \vol(B(\Scal^{d-1}))$. Using Lemma \ref{lem:spherical_cap_volume_upper_bound}, we get the desired bound on $N$.
\end{proof}

\begin{lemma}\label{lem:spherical_cap_volume_upper_bound}
    Suppose $d\geq 2$ and $h\in (0,1)$. Let $\vol(\Scal^{d-1})=S_{d-1}$. Then, for any $x\in \Scal^{d-1}$, we have,
    \[\vol(B(x,h))\leq \frac{2S_{d-1}}{\sqrt{d}}h^{d-1}.\]
\end{lemma}
\begin{proof}
    Let $\nu$ is the uniform distribution over $\Scal^{d-1}$.We first show that $\nu(B(x,h))\leq \frac{2}{\sqrt{d}}h^{d-1}$, and then use the fact that $\vol(B(x,h)) = \nu(B(x,h))S_{d-1}$. Towards bounding $\nu(B(x,h))$ from above, we can generate a random sample $\theta=(\theta_1,\ldots,\theta_d)\sim\nu$ by first drawing $Y_1,\ldots,Y_d$ independently from a standard normal distribution, and then taking
    $$ \theta_i=\frac{Y_i}{(Y_1^2+\cdots+Y_d^2)^{1/2}}.$$
    This works because the distribution $Y=(Y_1,\ldots,Y_d)$ is spherically symmetric. Now for any $k$,the sum $Y_1^2+\cdots,Y_d^2$ is a chi-squared distribution with $k$ degrees of freedom , denoted by $\chi^2(k)$. It is well known that if $A\sim\chi^2(k)$ and $B\sim\chi^2(l)$ are independent, then $A/(A+B)$ has a Beta$(k/2,l/2)$ distribution. Thus $\theta_1^2=Y_1^2/(Y_1^2+\cdots+Y_d^2)$ follows a Beta$(1/2,(d-1)/2)$ distribution.
    Without loss of generality, assume $x=e_1$. Then,
    \[\theta\in B(x,h)\iff \|\theta-e_1\|^2\leq h^2 \iff \theta_1\geq 1-h^2/2.\]
   Letting $\epsilon = 1-(1-h^2/2)^2=h^2(1-h^2/4)$ we get,
    \begin{eqnarray*}
        \nu(B(x,h))&=&\frac{1}{2}\cdot\pr\left(\theta_1^2\geq(1-(h^2/2))^2\right)=\frac{1}{2}\cdot\pr\left(\theta_1^2\geq1-\epsilon\right)\\
        &\leq&\frac{1}{2}\cdot\frac{1}{B(1/2,(d-1)/2)}\cdot\frac{\epsilon^{(d-1)/2}}{(d-1)/2}\cdot(1-\epsilon)^{-1/2}\\
        &=&\frac{1}{2}\cdot\frac{\Gamma(d/2)}{\Gamma(1/2)\Gamma((d-1)/2)(d-1)/2}\cdot h^{d-1}\cdot (1-h^2/4)^{(d-1)/2}\cdot\frac{1}{(1-h^2/2)}\\
        &\leq&\frac{\Gamma(d/2)}{\Gamma(1/2)\Gamma((d-1)/2)(d-1)/2}\cdot h^{d-1}\\
        &=&\frac{1}{\sqrt{\pi}(d-1)/2}\cdot\frac{\Gamma(d/2)}{\Gamma(d/2-1/2)}\cdot h^{d-1}\\
        &\leq&\frac{\sqrt{d/2}}{\sqrt{\pi}(d-1)/2}\cdot h^{d-1}\\
        &\leq&\frac{\sqrt{d/2}}{\sqrt{\pi}(d/4)}\cdot h^{d-1}\leq\frac{2}{\sqrt{d}}\cdot h^{d-1}
    \end{eqnarray*}
    where, the first inequality is due to Lemma \ref{lem:beta-mass-bound}, the second inequality is due to the fact that $h\in (0,1)$, the third inequality is due the fact that $\Gamma(x)/\Gamma(x-0.5)\leq\sqrt{x}$ and the fourth inequality is due to the condition on $d$.
\end{proof}

\begin{lemma}[Lemma~13~\cite{dasgupta2020expressivity}]
\label{lem:beta-mass-bound}
Suppose $Z$ has Beta$(\alpha,\beta)$ distribution with $\alpha\leq 1$ and $\beta\geq 1$. For any $0<\epsilon<1$,
\[ \frac{1}{B(\alpha,\beta)}\cdot\frac{\epsilon^{\beta}}{\beta}\leq\pr(Z\geq 1-\epsilon)\leq \frac{1}{B(\alpha,\beta)}\cdot\frac{\epsilon^{\beta}}{\beta}\cdot (1-\epsilon)^{\alpha-1}. \]
\end{lemma}
\section{Proofs from \cref{sec:mode_estimation}}\label{sec:mode_estimation_proofs}

\subsection{Proof of Theorem \ref{thm:single_mode_guarantee}}
\begin{proof}
    Let $x_0\in\Mcal$ and define \[r_n(x_0) = \inf\{r: B(x_0,r)\cap X_n\neq \emptyset\}.\]
    Let $\tau\in(0,1)$ be a constant to be fixed later. Assume for now that $r_n(x_0)\leq \frac{\tau}{2} r_{x_0}$. In due course, we will show that this holds with high probability. Also, consider $\rtilde$, to be appropriately chosen later, satisfying the following:
    \[\frac{2r_n(x_0)}{\tau} \leq \rtilde \leq r_{x_0}.\]
    Since $\rtilde\geq 2r_n(x_0)$, our goal is to show
    \[\sup_{x\in\Scal^{d-1}\setminus B(x_0,\rtilde)} \fhat_n(x) < \inf_{x\in B(x_0,r_n(x_0))} \fhat_n(x)\]
    as this will ensure that $\|\xhat_0-x\|\leq\rtilde$. To see this, note that if $\sup_{x\in B(x_0,\rtilde)\setminus B(x_0,r_n(x_0))} \fhat_n(x) > \sup_{x\in \Scal^{d-1}\setminus B(x_0,\rtilde)} \fhat_n(x)$, then $\xhat_0\in B(x_0,\rtilde)$. On the other hand, if $\sup_{x\in B(x_0,\rtilde)\setminus B(x_0,r_n(x_0))} \fhat_n(x) \leq \sup_{x\in \Scal^{d-1}\setminus B(x_0,\rtilde)} \fhat_n(x)$, then $\xhat_0\in B(x_0, r_n(x_0))\subset B(x_0, \rtilde)$.

    We first bound $\sup_{x\in \Scal^{d-1}\setminus B(x_0,\rtilde)} \fhat_n(x)$ from above. Using definition \ref{def:local_parameterization}, observe that, 
    \begin{equation}\label{eq:sup_upperbound}
        \sup_{x\in A_{x_0}\setminus B(x_0,\rtilde/2)} f(x) \leq f(x_0)-\Ccheck_{x_0}(\rtilde/2)^2 \triangleq \Fhat
    \end{equation}
    To apply Lemma \ref{lem:relative_density_bound}, note the following. For any $x\in \Scal^{d-1} \setminus B(x_0,\rtilde/2), f(x)\leq \Fhat$. This is because $A_{x_0}$ is a level set of unimodal $f$, that is, $\sup_{x\notin A_{x_0}}f(x) \leq \inf_{x\in A_{x_0}} f(x)$. Now, for any $x\in \Scal^{d-1}\setminus B(x_0,\rtilde)$, let $\epsilon = \Fhat-f(x)$. Then it is easy to see that $\rhat(\epsilon,x)\geq \rtilde/2$. Therefore, for every $j=1.,\ldots,m$, if we can somehow restrict $\diam(C_j)\leq \frac{\rtilde}{2}$, then it automatically ensures that $\diam(C_j)\leq \rhat(\epsilon,x)$.
    Applying Lemma \ref{lem:relative_density_bound}, the following holds with probability at least $1-\delta$,
    \begin{eqnarray}\label{eq:final_sup_upperbound}
        \sup_{x\in \Scal^{d-1}\setminus B(x_0,\rtilde)}\fhat_n(x) &\leq& \left(1+ C_{\delta}\sqrt{\frac{d\log m}{k}}\right)(f(x)+\epsilon) + \gamma_n \nonumber \\
        &=& \left(1+ C_{\delta}\sqrt{\frac{d\log m}{k}}\right)\Fhat + \gamma_n \nonumber \\
        &=& \left(1+ C_{\delta}\sqrt{\frac{d\log m}{k}}\right)(f(x_0)-\Ccheck_{x_0}(\rtilde/2)^2) + \gamma_n
    \end{eqnarray}

    Now we bound $\inf_{x\in B(x_0,r_n(x_0)}\fhat_n(x)$ from below. Since $\tau\rtilde\leq \tau r_{x_0}\leq r_{x_0}$, using definition \ref{def:local_parameterization}, observe that,
    \begin{equation}\label{eq:inf_lowerbound}
        \inf_{x\in B(x_0,\tau\rtilde)} f(x) \geq f(x_0)-\Chat_{x_0}(\tau\rtilde)^2 \triangleq \Fcheck
    \end{equation}
    Now for any $x\in B(x_0,r_n(x_0))$, let $\epsilon = f(x)-\Fcheck$. Then we have $\rcheck(\epsilon,x)\geq \tau\rcheck-r_n(x_0)\geq \tau\rtilde-\tau\rtilde/2 = \tau\rtilde/2$. 
    Therefore, for every $j=1.,\ldots,m$, if we can somehow restrict $\diam(C_j)\leq \frac{\tau\rtilde}{2}$, then it automatically ensures that $\diam(C_j)\leq \rcheck(\epsilon,x)$.
    Again, applying Lemma \ref{lem:relative_density_bound}, the following holds with probability at least $1-\delta$,

    \begin{eqnarray}\label{eq:final_inf_lowerbound}
    \inf_{x\in B(x_0,r_n(x_0))}\fhat_n(x)&\geq& \left(1- C_{\delta}\sqrt{\frac{d\log m}{k}}\right)(f(x)-\epsilon) + \gamma_n \nonumber \\
    &=& \left(1 - C_{\delta}\sqrt{\frac{d\log m}{k}}\right)\Fcheck - \gamma_n \nonumber \\
    &=& \left(1 - C_{\delta}\sqrt{\frac{d\log m}{k}}\right)(f(x_0)-\Chat_{x_0}(\tau\rtilde)^2) - \gamma_n
    \end{eqnarray}
    Now, we can pick $\tau$ and $\rtilde$ such that the r.h.s. of \eqref{eq:final_sup_upperbound} is less than the r.h.s. of \eqref{eq:final_inf_lowerbound}. It suffices to pick $\tau^2 = \frac{\Ccheck_{x_0}}{8\Chat_{x_0}}$ and $\rtilde^2\geq \frac{16}{\Ccheck_{x_0}}\left(f(x_0)C_{\delta}\sqrt{\frac{d\log m}{k}}+\gamma_n\right)$. Using Lemma \ref{lem:sample_size_bound} and the Theorem's assumption on $k$, it is easy to check that if $n\geq \frac{9m\log(m/\delta)\max\{\|f\|_{\infty},1\}}{S_{d-1}(f(x_0)C_{\delta})^2d\log m}$, then $\gamma_n$ is at most $f(x_0)C_{\delta}\sqrt{\frac{d\log m}{k}}$ and therefore $\rtilde^2\geq 32f(x_0)\frac{C_{\delta}}{\Ccheck_{x_0}}\sqrt{\frac{d\log m}{k}}$. Therefore, we are free to choose $\rtilde\geq \max\left\{\sqrt{32f(x_0)\frac{C_{\delta}}{\Ccheck_{x_0}}\sqrt{\frac{d\log m}{k}}}, \frac{2r_n(x_0)}{\tau}\right\}$.

    All that remains to show is that, (i) $r_n(x_0)\leq \frac{\tau}{2}r_{x_0}$ with high probability, and (ii) $\diam(C_j)\leq \min\{\rtilde/2,\tau\rtilde/2\}$ for every $j = 1.,\ldots,m$.
    Let us focus on the first of the two. Let $r= \sqrt{32f(x_0)\frac{C_{\delta}}{\Ccheck_{x_0}}\sqrt{\frac{d\log m}{k}}}$. The Theorem's assumption on $k$ ensures that $r\leq r_{x_0}$. Next, we show that with high probability, $|B(x_0,\tau r/2)\cap X_n|\geq \frac{k}{n}>0$ which implies that $r_n(x_0)\leq \tau r/2$. We first bound the probability mass of $f(B(x_0,\tau r/2))$ from below as follows:
    \begin{eqnarray}\label{eq:ball_mass_lower_bound}
        f(B(x_0,\tau r/2)) &\geq& \vol(B(x_0,\tau r/2))\inf_{x\in B(x_0,\tau r/2)} f(x) \nonumber \\
        &\stackrel{a}{\geq}& \frac{1}{3\sqrt{d}} (3/4)^{\frac{d-1}{2}}S_{d-1}(\tau r/2)^{d-1} \inf_{x\in B(x_0,\tau r/2)} f(x) \nonumber \\
        &=&\frac{1}{3\sqrt{d}} (3/4)^{\frac{d-1}{2}}S_{d-1} \left(f(x_0)\frac{C_{\delta}}{\Chat_{x_0}}\right)^{\frac{d-1}{2}}\left(\frac{d\log m}{k}\right)^{\frac{d-1}{4}}\inf_{x\in B(x_0,\tau r/2)} f(x) \nonumber \\
        &\stackrel{b}{\geq}&\frac{1}{3\sqrt{d}} S_{d-1} \left(\frac{3f(x_0)}{4}\frac{C_{\delta}}{\Chat_{x_0}}\right)^{\frac{d-1}{2}}\left(\frac{d\log m}{k}\right)^{\frac{d-1}{4}}\left(f(x_0)-\Chat_{x_0}(\tau r/2)^2\right) \nonumber \\
        &= & \frac{1}{3\sqrt{d}} S_{d-1} \left(\frac{3f(x_0)}{4}\frac{C_{\delta}}{\Chat_{x_0}}\right)^{\frac{d-1}{2}}\left(\frac{d\log m}{k}\right)^{\frac{d-1}{4}}f(x_0)\left(1-C_{\delta}\sqrt{\frac{d\log m}{k}}\right) \nonumber \\
        &\geq & \frac{1}{6\sqrt{d}} S_{d-1} \left(\frac{3f(x_0)}{4}\frac{C_{\delta}}{\Chat_{x_0}}\right)^{\frac{d-1}{2}}\left(\frac{d\log m}{k}\right)^{\frac{d-1}{4}}f(x_0) \nonumber \\
        & =& \frac{S_{d-1}}{6\sqrt{d}}\left(f(x_0)\right)^{\frac{d+1}{2}}\left(\frac{3C_{\delta}}{4\Chat_{x_0}}\right)^{\frac{d-1}{2}}\left(\frac{d\log m}{k}\right)^{\frac{d-1}{4}}
    \end{eqnarray}
where, inequality $a$ follows from Lemma \ref{lem:mass-to-radius} and the fact that $\nu(C) = \frac{\vol(C)}{S_{d-1}}$ for any $C\subset\Scal^{d-1}$. Inequality $b$ follows from Definition \ref{def:local_parameterization}.

If the r.h.s. of \eqref{eq:ball_mass_lower_bound} is at least $\frac{k}{n}+\frac{C_{\delta}\sqrt{kd\log n}}{n}$, then using Lemma 7 of \cite{chaudhuri2010cluster_tree}, with probability at least $1-\delta$, $|B(x_0,\tau r/2)\cap X_n|\geq k$. 
Therefore, $r_n(x_0)\leq \tau r/2$. This happens whenever $k\geq C_{\delta}^2d\log n$ and $k\leq \left(\frac{nS_{d-1}}{12\sqrt{d}}\right)^{\frac{4}{d+3}}\left(f(x_0)\right)^{\frac{2d+2}{d+3}}\left(\left(\frac{3C_{\delta}}{4\Chat_{x_0}}\right)^2d\log m\right)^{\frac{d-1}{d+3}}$.
Further, using Theorem's assumption on $k$, it follows that, $r_n(x_0)\leq \tau r/2\leq \tau r_{x_0}/2$. It is now clear that we can choose $\rtilde=r$. Once the choice of $\rtilde$ is finalized, using the fact that $\tau <1$, Lemma \ref{lem:diameter_bound} and the Theorem's assumption on $k$, it is immediate that $\diam(C_j)\leq \tau\rtilde/2$, for every $j=1,\ldots,m$.
\end{proof}

\subsection{Proof of Theorem \ref{thm: multiple_modes_guarantee}}

Since the proof is a bit involved, before presenting the detailed proof, we first provide a sketch of proof which is easier to follow.
\paragraph{Proof sketch:} 
We define \[r_n(x_0) = \inf\{r: B(x_0,r)\cap X_n\neq \emptyset\}\]
and choose $\tau\in(0,1)$ and $\rtilde$  satisfying the following:
\[\frac{2r_n(x_0)}{\tau} \leq \rtilde \leq \rbar \leq r_{x_0}\]
Then, $B(x_0,\rtilde)\subset B(x_0,\rbar)\subset A_{x_0}$. 
It is clear that if Alg. \ref{alg:multiple_mode_estimation} picks a maximizer of $\fhat_n$ from a set contained in $A_{x_0}$ and intersecting $B(x_0, r_n(x)))$, then this maximizer must be within $\rtilde$ of $x_0$. If $\rtilde$ is appropriately chosen, then that would imply the result. Thus, our proof strategy  is to show that Alg. \ref{alg:multiple_mode_estimation} indeed picks a maximizer from such a set.  To this end, let $\lambda\geq \inf_{x\in B(x_0,r_n(x_0))}\fhat_n(x)$ be the first level in the iteration of Alg. \ref{alg:multiple_mode_estimation} containing a point from $B(x_0,r_n(x_0))$. Also, let $S$ be the set that $r$-separates $A_{x_0}$ from $\Scal^{d-1}_{\lambda_{x_0}}\setminus A_{x_0}$. Define
    \[\Scal^{d-1}_{S\rightsquigarrow x_0} \triangleq \{x: \exists \mbox{ a path } \Pcal \mbox{ from } x \mbox{ to } x_0, \Pcal \cap S =\emptyset \}\]
to be the set of points reachable from $x_0$ without crossing $S$. Clearly, $\Scal^{d-1}_{S\rightsquigarrow x_0} \supseteq A_{x_0}$ and by Definition \ref{def:local_parameterization}, $\Scal^{d-1}_{S\rightsquigarrow x_0} \cap \Scal^{d-1}_{\lambda_{x_0}} = A_{x_0}$. The crux of our proof strategy is to show that, (i) at the level $\lambda-\tilde{\epsilon}$, $G(\lambda-\tilde{\epsilon})$ contains no point from $\Scal^{d-1}_{S\rightsquigarrow x_0}\setminus B(x_0,\rbar)$, and (ii) there does not exist any edge between any $x\in B(x_0,\rbar)$ and points in $\Scal^{d-1}\setminus \Scal^{d-1}_{S\rightsquigarrow x_0}$ in $G(\lambda-\tilde{\epsilon})$. Therefore, it follows that there exists a CC $\tilde{A}_j$ of $G(\lambda-\tilde{\epsilon})$ that is disconnected from $\Scal^{d-1}\setminus\Scal^{d-1}_{S\rightsquigarrow x_0}$, and thus any previous mode found by Alg. \ref{alg:multiple_mode_estimation} such that $\tilde{A}_j\cap X_{n,\lambda}$ contains only points from $A_{x_0}$ and intersects $B(x_0,r_n(x_0))$. Therefore, the procedure picks a point 
$\hat{x}_0 = \argmax_{x\in \tilde{A}_j\cap X_{n,\lambda}}\fhat_n(x)$ satisfying $\|\hat{x}_0-x_0\|\leq\rtilde$.

To achieve the first goal, it is enough to show that,
\begin{equation}\label{eqn:multiple_modes_proof_crux}
    \sup_{x\in \Scal^{d-1}_{S\rightsquigarrow x_0}\setminus B(x_0,\rbar)}\fhat_n(x) \leq \inf_{x\in B(x_0,r_n(x_0))}\fhat_n(x) -\gamma_n
\end{equation}
since for $\tilde{\epsilon}\leq\gamma_n$, that would imply:
\begin{align*}
    \sup_{x\in \Scal^{d-1}_{S\rightsquigarrow x_0}\setminus B(x_0,\rbar)}\fhat_n(x) &\leq \inf_{x\in B(x_0,r_n(x_0))}\fhat_n(x) -\gamma_n \leq \inf_{x\in B(x_0,r_n(x_0))}\fhat_n(x) -\tilde{\epsilon}\leq \lambda - \tilde{\epsilon}
\end{align*}
To prove \eqref{eqn:multiple_modes_proof_crux}, we repeatedly invoke Lemma \ref{lem:relative_density_bound} with different $\epsilon$ values. In particular, we find the upper bound of the l.h.s. of \eqref{eqn:multiple_modes_proof_crux} by invoking Lemma \ref{lem:relative_density_bound} with properly tuned $\epsilon$, and show that this is less than the lower bound of the r.h.s. of \eqref{eqn:multiple_modes_proof_crux} obtained by invoking again Lemma \ref{lem:relative_density_bound} with properly tuned (but different than before) $\epsilon$.

To achieve the second goal, we ensure that for any point $x\in B(x_0,\rbar)$, the ball $B(x, \rtilde/2)$ contains at least $k$ points such that $r_k(x)\leq \rtilde/2$. Further, appropriately setting $\rtilde$, in comparison to separation parameter $r$ and $\alpha$ given in Definition \ref{def:density_graph}, we ensure that there can not be an edge between $x\in B(x_0,\rbar)$ and points in $\Scal^{d-1}\setminus \Scal^{d-1}_{S\rightsquigarrow x_0}$ in $G(\lambda-\tilde{\epsilon})$.
\qed

Now we present the actual proof. 
\begin{proof}
    Let $r_n(x_0) = \inf\{r: B(x_0,r)\cap X_n\neq \emptyset\}$ be the smallest radius around $x_0$ containing a sample from $X_n$. Similar to the proof of Theorem \ref{thm:single_mode_guarantee}, let $\tau\in(0,1)$ be a constant to be fixed later. Assume for now that $r_n(x_0)\leq \frac{\tau}{2} r_{x_0}$. In due course, we will show that this holds with high probability. Also, consider $\rtilde$, to be appropriately chosen later, satisfying the following:
    \[\frac{2r_n(x_0)}{\tau} \leq \rtilde \leq r_{x_0}.\]

Here, in addition we use a new scale $\rbar = r_{x_0}/2$ satisfying the following:
\[\frac{2r_n(x_0)}{\tau} \leq \rtilde \leq \rbar \leq r_{x_0}.\]
Clearly, $B(x_0,\rbar)\subset B(x_0,r_{x_0}\subset A_{x_0}$. Note that, for appropriate choice of $\rtilde$, if the maximizer of $\fhat_n$ comes from a set contained in $A_{x_0}$ and intersecting $B(x_0,r_n(x_0))$, then that maximizer must be within $\rtilde$ of $x_0$. Next, we show that Alg. \ref{alg:multiple_mode_estimation} indeed picks a maximizer out of such a set.

To this end, let $\lambda\geq \inf_{x\in B(x_0,r_n(x_0))}\fhat_n(x)$ be the first level in the iteration of Alg. \ref{alg:multiple_mode_estimation} containing a point from $B(x_0,r_n(x_0))$. Also, let $S$ be the set that $r$-separates $A_{x_0}$ from $\Scal^{d-1}_{\lambda_{x_0}}\setminus A_{x_0}$. Define
    \[\Scal^{d-1}_{S\rightsquigarrow x_0} \triangleq \{x: \exists \mbox{ a path } \Pcal \mbox{ from } x \mbox{ to } x_0, \Pcal \cap S =\emptyset \}\]
to be the set of points reachable from $x_0$ without crossing $S$. Clearly, $\Scal^{d-1}_{S\rightsquigarrow x_0} \supseteq A_{x_0}$ and by Definition \ref{def:local_parameterization}, $\Scal^{d-1}_{S\rightsquigarrow x_0} \cap \Scal^{d-1}_{\lambda_{x_0}} = A_{x_0}$.

Our goal is to show that, (i)  at the level $\lambda-\tilde{\epsilon}$, $G(\lambda-\tilde{\epsilon})$ does not contain any point from $\Scal^{d-1}_{S\rightsquigarrow x_0}\setminus B(x_0,\rbar)$, and (ii) there does not exist any edge between any $x\in B(x_0,\rbar)$ and points in $\Scal^{d-1}\setminus \Scal^{d-1}_{S\rightsquigarrow x_0}$ in $G(\lambda-\tilde{\epsilon})$. Whenever the above goals are satisfied, it follows that there exists a CC $\tilde{A}_j$ of $G(\lambda-\tilde{\epsilon})$ that is disconnected from $\Scal^{d-1}\setminus\Scal^{d-1}_{S\rightsquigarrow x_0}$, and thus is disconnected from any previous mode found by Alg. \ref{alg:multiple_mode_estimation} such that $\tilde{A}_j\cap X_{n,\lambda}$ contains only points from $A_{x_0}$ and intersects $B(x_0,r_n(x_0))$. Therefore, Alg. \ref{alg:multiple_mode_estimation} picks a point 
$\hat{x}_0 = \argmax_{x\in \tilde{A}_j\cap X_{n,\lambda}}\fhat_n(x)$ satisfying $\|\hat{x}_0-x_0\|\leq\rtilde$.

Rest of the proof ensures that the above two goals are satisfied. We start with the first goal. Let $\bar{F} = f(x_0)-\Ccheck_{x_0}(\rbar/2)^2$. Note that $\forall x \in \Scal^{d-1}_{S\rightsquigarrow x_0}\setminus B(x_0,\rbar)$ we have $\sup_{x'\in B(x,\rtilde/2)}f(x') \leq  \bar{F}$. To see this, first observe that, 
\[B(x,\rtilde/2)\subset \left(\Scal^{d-1}_{S\rightsquigarrow x_0}\setminus B(x_0,\rbar/2)\right)\cup S_r\]
as long as $\rtilde/2\leq r$. We separately show that $\sup_{x\in \Scal^{d-1}_{S\rightsquigarrow x_0}\setminus B(x_0,\rbar/2)}f(x) \leq \bar{F}$ and $\sup_{x\in S_r}f(x) \leq \bar{F}$. The set $\Scal^{d-1}_{S\rightsquigarrow x_0}\setminus B(x_0,\rbar/2)$ can be written as $\Scal^{d-1}_{S\rightsquigarrow x_0}\setminus B(x_0,\rbar/2) = A_{x_0}\setminus B(x_0,\rbar/2) \cup \Scal^{d-1}_{S\rightsquigarrow x_0}\setminus A_{x_0}$. It follows immediately from Definition \ref{def:local_parameterization} that $\sup_{x\in A_{x_0}\setminus B(x_0,\rbar/2)} \leq \bar{F}$. Also, from the definition of $\Scal^{d-1}_{S\rightsquigarrow x_0}$, it follows that $\sup_{x\in \Scal^{d-1}_{S\rightsquigarrow x_0}\setminus A_{x_0}}f(x) < \lambda_{x_0}\leq \bar{F}$. Next, by Definition \ref{def:r_separation}, $\sup_{x\in S_r}f(x) <\lambda_{x_0}\leq \bar{F}$.

For any $x\in \Scal^{d-1}_{S\rightsquigarrow x_0}\setminus B(x_0,\rbar)$, let $\epsilon = \bar{F}-f(x)$. Then it is easy to see that $\rhat(\epsilon,x)\geq \rbar/2\geq \rtilde/2$. Therefore, for every $j=1.,\ldots,m$, if we can somehow restrict $\diam(C_j)\leq \frac{\rtilde}{2}$, then it automatically ensures that $\diam(C_j)\leq \rhat(\epsilon,x)$.
    Applying Lemma \ref{lem:relative_density_bound}, the following holds with probability at least $1-\delta$,
    \begin{eqnarray}\label{eq:final_sup_upperbound_multiple_mode}
        \sup_{x\in \Scal^{d-1}_{S\rightsquigarrow x_0}\setminus B(x_0,\rbar)}\fhat_n(x) &\leq& \left(1+ C_{\delta}\sqrt{\frac{d\log m}{k}}\right)(f(x)+\epsilon) + \gamma_n \nonumber \\
        &=& \left(1+ C_{\delta}\sqrt{\frac{d\log m}{k}}\right)\bar{F} + \gamma_n \nonumber \\
        &=& \left(1+ C_{\delta}\sqrt{\frac{d\log m}{k}}\right)(f(x_0)-\Ccheck_{x_0}(\rbar/2)^2) + \gamma_n \nonumber \\
        &\leq& \left(1+ C_{\delta}\sqrt{\frac{d\log m}{k}}\right)(f(x_0)-\Ccheck_{x_0}(\rtilde/2)^2) + \gamma_n
    \end{eqnarray}

    Now, for any $x\in S_r$, by $r$-separation, $f(x)\leq\inf_{x'\in A_{x_0}}f(x')\leq \bar{F}$. Next for any $x\in S_{r/2}$, let $\epsilon = \bar{F}-f(x)$. Then it is easy to see that $\rhat(\epsilon,x)\geq r/2\geq\rtilde/2$, provided $\rtilde\leq r$. Then using the same argument as before the following holds with probability at least $1-\delta$,

    \begin{equation}\label{eqn:separor_sup_upper_bound_multiple_mode}
        \sup_{x\in S_{r/2}}\fhat_n(x) \leq \left(1+ C_{\delta}\sqrt{\frac{d\log m}{k}}\right)(f(x_0)-\Ccheck_{x_0}(\rtilde/2)^2) + \gamma_n
    \end{equation}
    
    Now we bound $\inf_{x\in B(x_0,r_n(x_0)}\fhat_n(x)$ from below. Since $\tau\rtilde\leq \tau r_{x_0}\leq r_{x_0}$, using definition \ref{def:local_parameterization}, observe that,
    \begin{equation}\label{eq:inf_lowerbound}
        \inf_{x\in B(x_0,\tau\rtilde)} f(x) \geq f(x_0)-\Chat_{x_0}(\tau\rtilde)^2 \triangleq \Fcheck
    \end{equation}
    Now for any $x\in B(x_0,r_n(x_0))$, let $\epsilon = f(x)-\Fcheck$. Then we have $\rcheck(\epsilon,x)\geq \tau\rcheck-r_n(x_0)\geq \tau\rtilde-\tau\rtilde/2 = \tau\rtilde/2$. 
    Therefore, for every $j=1.,\ldots,m$, if we can somehow restrict $\diam(C_j)\leq \frac{\tau\rtilde}{2}$, then it automatically ensures that $\diam(C_j)\leq \rcheck(\epsilon,x)$.
    Again, applying Lemma \ref{lem:relative_density_bound}, the following holds with probability at least $1-\delta$,

    \begin{eqnarray}\label{eq:final_inf_lowerbound_multiple_modes}
    \inf_{x\in B(x_0,r_n(x_0))}\fhat_n(x)&\geq& \left(1- C_{\delta}\sqrt{\frac{d\log m}{k}}\right)(f(x)-\epsilon) + \gamma_n \nonumber \\
    &=& \left(1 - C_{\delta}\sqrt{\frac{d\log m}{k}}\right)\Fcheck - \gamma_n \nonumber \\
    &=& \left(1 - C_{\delta}\sqrt{\frac{d\log m}{k}}\right)(f(x_0)-\Chat_{x_0}(\tau\rtilde)^2) - \gamma_n
    \end{eqnarray}

Noting that $\tilde{\epsilon}\leq\gamma_n$, we can pick $\tau$ and $\rtilde$ such that 
\[\sup_{x\in \Scal^{d-1}_{S\rightsquigarrow x_0}\setminus B(x_0,\rbar)}\fhat_n(x) \leq \inf_{x\in B(x_0,r_n(x_0))}\fhat_n(x) -\gamma_n\leq \inf_{x\in B(x_0,r_n(x_0))}\fhat_n(x) -\tilde{\epsilon}\leq \lambda - \tilde{\epsilon}\]
and 
\[\sup_{x\in S_{r/2}}\fhat_n(x) \leq \inf_{x\in B(x_0,r_n(x_0))}\fhat_n(x) -\gamma_n\leq \inf_{x\in B(x_0,r_n(x_0))}\fhat_n(x) -\tilde{\epsilon}\leq \lambda - \tilde{\epsilon}\]
ensuring that the graph $G(\lambda-\tilde{\epsilon})$ contains no points from $\Scal^{d-1}_{S\rightsquigarrow x_0}\setminus B(x_0,\rbar)$ and no pints from $S_{r/2}$. 
It suffices to pick $\tau^2 = \frac{\Ccheck_{x_0}}{8\Chat_{x_0}}$ and $\rtilde^2\geq \frac{16}{\Ccheck_{x_0}}\left(f(x_0)C_{\delta}\sqrt{\frac{d\log m}{k}}+3\gamma_n/2\right)$. Using Lemma \ref{lem:sample_size_bound} and the Theorem's assumption on $k$, it is easy to check that if $n\geq \frac{9m\log(m/\delta)\max\{\|f\|_{\infty},1\}}{S_{d-1}(f(x_0)C_{\delta})^2d\log m}$, then $\gamma_n$ is at most $f(x_0)C_{\delta}\sqrt{\frac{d\log m}{k}}$ and therefore $\rtilde^2\geq 40f(x_0)\frac{C_{\delta}}{\Ccheck_{x_0}}\sqrt{\frac{d\log m}{k}}$. Therefore, we are free to choose $\rtilde\geq \max\left\{\sqrt{40f(x_0)\frac{C_{\delta}}{\Ccheck_{x_0}}\sqrt{\frac{d\log m}{k}}}, \frac{2r_n(x_0)}{\tau}\right\}$.

Let's set $\rtilde = \sqrt{40f(x_0)\frac{C_{\delta}}{\Ccheck_{x_0}}\sqrt{\frac{d\log m}{k}}}$. 

Next, for any $x\in B(x_0,\rbar)$, the ball $B(x,\rtilde/2)$ is fully contained in $A_{x_0}$ since $\rbar+\rtilde/2\leq3\rbar/2=3r_{x_0}/4<r_{x_0}$. We seek to bound $f(B(x,\rtilde/2))$ from below as follows

\begin{eqnarray}\label{eq:ball_mass_lower_bound_multiple_mode}
        f(B(x,\rtilde/2)) &\geq& f(B(x,\tau\rtilde/2)) \nonumber \\
        &\geq& \vol(B(x,\tau\rtilde/2))\inf_{x\in B(x,\tau\rtilde/2)} f(x) \nonumber \\
        &\stackrel{a}{\geq}& \frac{1}{3\sqrt{d}} (3/4)^{\frac{d-1}{2}}S_{d-1}(\tau\rtilde/2)^{d-1} \inf_{x\in B(x,\tau\rtilde/2)} f(x) \nonumber \\
        &\geq& \frac{1}{3\sqrt{d}} (3/4)^{\frac{d-1}{2}}S_{d-1}(\tau\rtilde/2)^{d-1} \lambda_{x_0} \nonumber \\
        &\stackrel{b}{\geq}& \frac{S_{d-1}}{3\sqrt{d}}\left(\frac{15\lambda_{x_0}C_{\delta}}{16\Chat_{x_0}}\right)^{\frac{d-1}{2}}\left(\frac{d\log m}{k}\right)^{\frac{d-1}{4}}
    \end{eqnarray}
where, inequality $a$ follows from Lemma \ref{lem:mass-to-radius} and the fact that $\nu(C) = \frac{\vol(C)}{S_{d-1}}$ for any $C\subset\Scal^{d-1}$. Inequality $b$ follows from the fact that $f(x_0)\geq\lambda_{x_0}$. If the r.h.s. of \eqref{eq:ball_mass_lower_bound_multiple_mode} is at least $\frac{k}{n}+\frac{C_{\delta}\sqrt{kd\log n}}{n}$, then using Lemma 7 of \cite{chaudhuri2010cluster_tree}, with probability at least $1-\delta$, $|B(x_0,\rtilde/2)\cap X_n|\geq k$. Therefore, $r_k(x)\leq \rtilde/2$. This happens whenever $k\geq C_{\delta}^2d\log n$ and $k\leq \left(\frac{S_{d-1}}{6\sqrt{d}}\left(\frac{15f(x_0)C_{\delta}}{64\Ccheck_{x_0}}\right)^2 d\log m\right)^{\frac{d-1}{d+3}}n^{\frac{4}{d+3}}$. If in addition $\rtilde\leq r/\alpha$, then it is easy to see that $\alpha r_k(x)\leq \alpha\rtilde/2\leq r/2$. Therefore, $B(x,\alpha r_k(x))\subset \Scal^{d-1}_{S\rightsquigarrow x_0}$ and there can not be any edge between $x\in B(x_0,\rbar)$ and points in $\Scal^{d-1}\setminus \Scal^{d-1}_{S\rightsquigarrow x_0}$ in $G(\lambda-\tilde{\epsilon})$.

All that remains to show is, (i) $\rtilde\leq \rbar =r_{x_0}/2$, (ii) $r_n(x_0)\leq \tau\rtilde/2\leq \tau r_{x_0}/2$, (iii) $\rtilde\leq r/\alpha$, and (iv) $\diam (C_j)\leq\min\{\rtilde/2,\tau\rtilde/2\}$. If $k\geq \left(\frac{40f(x_0)C_{\delta}}{\Ccheck_{x_0}(\min\{r_{x_0}/2,r/\alpha\})^2}\right)^2d\log m$, using the value of $\rtilde$, the first and the third conditions are satisfied. Note that \eqref{eq:ball_mass_lower_bound_multiple_mode} hold for any $x\in B(x_0,\rbar)$, in particular, for $x_0$. Therefore, $r_n(x_0)\leq\tau\rtilde/2$ and the second condition is satisfied. Finally, using the fact that $\tau<1$, Lemma \ref{lem:diameter_bound} it is immediate that $\diam(C_j)\leq\tau\rtilde/2$ for every $j=1,\ldots,m$ whenever $k\leq \frac{1}{6}\left(\left(\frac{3}{16}f(x_0)\frac{C_{\delta}}{\Chat_{x_0}}\right)^2\log m\right)^{\frac{d-1}{d+3}} m^{\frac{4}{d+3}}$.

\end{proof}

\subsection{Auxiliary proofs}
\begin{lemma}\label{lem:diameter_bound}
    Let $x_0\in\Mcal$. Set $\tau^2 = \frac{\Ccheck_{x_0}}{8\Chat_{x_0}}$, $\rtilde^2_1 = 32f(x_0)\frac{C_{\delta}}{\Ccheck_{x_0}}\sqrt{\frac{d\log m}{k}}$, and  $\rtilde^2_2 = 40f(x_0)\frac{C_{\delta}}{\Ccheck_{x_0}}\sqrt{\frac{d\log m}{k}}$. Assume that the conclusion of the Lemma \ref{lem:inner-outer-balls} holds. Then, $\diam(C_j)\leq \tau\rtilde_1/2$ for every $j=1,\ldots,m$, provided the following holds.
    \[k\leq \frac{1}{6}\left(\left(\frac{3}{16}f(x_0)\frac{C_{\delta}}{\Chat_{x_0}}\right)^2\log m\right)^{\frac{d-1}{d+3}} m^{\frac{4}{d+3}}.\]

    Similarly, $\diam(C_j)\leq \tau\rtilde_2/2$ for every $j=1,\ldots,m$, provided the following holds.
    \[k\leq \frac{1}{6}\left(\left(\frac{5}{64}\lambda_{x_0}\frac{C_{\delta}}{\Chat_{x_0}}\right)^2\log m\right)^{\frac{d-1}{d+3}} m^{\frac{4}{d+3}}.\]
\end{lemma}
\begin{proof}
Since conclusion of \cref{lem:inner-outer-balls} holds, we can apply \cref{lem:general-diameter-bound} to get
\[\diam(C_j) \leq \frac{4}{\sqrt{3}}\left(\frac{6\sqrt{d}k}{m}\right)^{\frac{1}{(d-1)}}\]
Note that, $\frac{\tau\rtilde_1}{2} = \sqrt{f(x_0)\frac{C_{\delta}}{\Chat_{x_0}}}\left(\frac{d\log m}{k}\right)^{1/4}$. Then $\diam(C_j)\leq \tau\rtilde_1/2$ for every $j=1,\ldots,m$ provided
\begin{align*}
        &\frac{4}{\sqrt{3}}\left(\frac{6\sqrt{d}k}{m}\right)^{\frac{1}{(d-1)}} \leq \sqrt{f(x_0)\frac{C_{\delta}}{\Chat_{x_0}}}\left(\frac{d\log m}{k}\right)^{1/4} \\
         \Longrightarrow ~~ & \frac{16}{3}\left(\frac{6\sqrt{d}k}{m}\right)^{\frac{2}{d-1}} \leq f(x_0)\frac{C_{\delta}}{\Chat_{x_0}}\sqrt{\frac{d\log m}{k}} \\
         \Longrightarrow ~~ & \frac{256}{9}\left(\frac{6\sqrt{d}k}{m}\right)^{\frac{4}{d-1}} \leq \left(f(x_0)\frac{C_{\delta}}{\Chat_{x_0}}\right)^2\frac{d\log m}{k} \\
         \Longrightarrow ~~ & k^{\frac{d+3}{d-1}} \leq \left(\left(\frac{3}{16}f(x_0)\frac{C_{\delta}}{\Chat_{x_0}}\right)^2\log m\right) d^{\frac{d-3}{d-1}} m^{\frac{4}{d-1}} /6^{\frac{4}{d-1}}\\
        \Longrightarrow ~~ & k\leq \left(\left(\frac{3}{16}f(x_0)\frac{C_{\delta}}{\Chat_{x_0}}\right)^2\log m\right)^{\frac{d-1}{d+3}} d^{\frac{d-3}{d+3}} m^{\frac{4}{d+3}} /6^{\frac{4}{d+3}}
\end{align*}
Noting that, $d^{\frac{d-3}{d+3}} = d^{1-\frac{6}{d+3}}$ is an increasing function of $d$, we have $d^{\frac{d-3}{d+3}}\geq \frac{1}{2^{1/5}}$ for $d\geq 2$. Similarly, noting that $6^{\frac{4}{d+3}}$ is a decreasing function of $d$, we have $6^{\frac{4}{d+3}}\leq 6^{4/5}$ for $d\geq 2$. Thus, $d^{\frac{d-3}{d+3}}/6^{\frac{4}{d+3}}\geq \frac{1}{2^{1/5}}\cdot\frac{1}{6^{\frac{4}{5}}}=\frac{1}{2^{1/5}}\cdot\frac{1}{2^{4/5}3^{4/5}}=\frac{1}{2}\cdot\frac{1}{3^{4/5}}\geq \frac{1}{2}\cdot\frac{1}{3}=\frac{1}{6}$. Plugging in this value, the choice of $k$ as given in the Lemma statement suffices.

Similarly, note that, $\frac{\tau\rtilde_2}{2} = \sqrt{\frac{5}{4}f(x_0)\frac{C_{\delta}}{\Chat_{x_0}}}\left(\frac{d\log m}{k}\right)^{1/4} \geq \sqrt{\frac{5\lambda_{x_0}C_{\delta}}{4\Chat_{x_0}}}\left(\frac{d\log m}{k}\right)^{1/4}$. Using similar steps as above, the result follows.
\end{proof}

\begin{lemma}\label{lem:sample_size_bound}
    Let $x_0\in \Mcal$, $k\geq \frac{(f(x_0)C_{\delta})^2}{9}d\log m$, and  $\gamma_n = \alpha_n/(S_{d-1}(k/m))$, where $\alpha_n = 2\sqrt{\frac{kS_{d-1}\|f\|_{\infty}\log(m/\delta)}{mn}}+\frac{2\log (m/\delta)}{3n}$. Then,  $\gamma_n \leq f(x_0)C_{\delta}\sqrt{\frac{d\log m}{k}}$ provided the following holds, 
    \[n \geq \frac{9m\log(m/\delta)\max\{\|f\|_{\infty},1\}}{S_{d-1}(f(x_0)C_{\delta})^2d\log m}.\]
\end{lemma}
\begin{proof}
    Let $\beta = \frac{m\log(m/\delta)}{S_{d-1}kn}$. Then,
    \begin{align*}
        \gamma_n & = \frac{\alpha_n}{S_{d-1}k/m} = \frac{\left(2\sqrt{\frac{kS_{d-1}\|f\|_{\infty}\log(m/\delta)}{mn}}+\frac{2\log (m/\delta)}{3n}\right)}{S_{d-1}k/m}\\
        & = 2 \sqrt{\beta\|f\|_{\infty}} + \frac{2}{3}\beta \\
        & \leq  2 \sqrt{\beta\max\{\|f\|_{\infty},1\}} + \frac{2}{3}\beta\max\{\|f\|_{\infty},1\} \\
        & \stackrel{a}{\leq} 3 \sqrt{\beta\max\{\|f\|_{\infty},1\}} \\
        & \stackrel{b}{\leq} f(x_0)C_{\delta}\sqrt{\frac{d\log m}{k}}
    \end{align*}
    Inequality $a$ holds, if $\beta\max\{\|f\|_{\infty},1\} \leq 1$, or in other words,
    \begin{equation}\label{eqn:sample_size_bound_1}
        n \geq \frac{m\log(m/\delta)\max\{\|f\|_{\infty},1\}}{kS_{d-1}}
    \end{equation}

    Inequality $b$ holds, if, 
    \begin{equation}\label{eqn:sample_size_bound_2}
        n\geq \frac{9m\log(m/\delta)\max\{\|f\|_{\infty},1\}}{S_{d-1}(f(x_0)C_{\delta})^2d\log m}
    \end{equation}
    Therefore, both the inequalities hold, if, 
    \[n\geq \frac{m\log(m/\delta)\max\{\|f\|_{\infty},1\}}{S_{d-1}}\times \max\left\{\frac{1}{k}, \frac{9}{(f(x_0)C_{\delta})^2d\log m}\right\}.\]
    Since $k\geq \frac{(f(x_0)C_{\delta})^2}{9}d\log m$, choosing $n$ as given in the Lemma statement suffices.
\end{proof}
\section{Empirical evaluations details}\label{sec:empirical_evaluations_extra}
Since our density and mode estimation results hold for data lying on the unit sphere, for our empirical evaluations, we choose distributions that are supported on the unit sphere. A well known distribution defined on the unit sphere is von Mises-Fisher distribution which is a directional distribution on the surface of the unit sphere $\mathcal{S}^{d-1}$. For any $x\in\mathcal{S}^{d-1}$, the probability density of a von Mises-Fisher distribution is given as in \eqref{eq:von_mises_fisher}
where, $\mu\in\mathcal{S}^{d-1}$ is the mean direction, $\kappa>0$ is the concentration parameter, $d$ is the dimension, and $I$ is the modified Bessel function of the first kind.  The distribution becomes narrower around $\mu$ with increasing $\kappa$. In particular, the reciprocal value $1/\kappa$ resembles the variance parameter of the normal distribution. We chose a mixture of two von Mises-Fisher distributions given in \eqref{eq:von_mises_fisher_mixture}
where, for $i=1,2$,  $f_i(x)$ is specified, using \eqref{eq:von_mises_fisher} by mean direction $\mu_i$, concentration parameter $\kappa_i$ and $w\in(0,1)$ is the mixture coefficient. In Fig.\ref{fig:vonmises}, we visually depict a mixture of von Mises-Fisher distribution in $\mathcal{S}^2$, where $\mu_1 = \left(-\sqrt{0.9},-\sqrt{0.1},0\right)\in\mathcal{S}^2$, $\kappa_1=10$, $\mu_2 = \left(-\sqrt{0.01}, -\sqrt{0.99}, 0\right)\in\mathcal{S}^2$, $\kappa_2 = 5$, and $w = 0.3$.

\subsection{Data generation process and computing environment}
For our empirical evaluations, we generate samples from a mixture of von Moses-Fisher distribution as given in \eqref{eq:von_mises_fisher_mixture}. Specifically, we use \texttt{scipy}\footnote{\url{https://docs.scipy.org/doc/scipy/reference/generated/scipy.stats.vonmises_fisher.html}} to generate samples from an individual von Mises-Fisher distribution. To generate samples from a mixture of two von Mises-Fisher distributions we first need to specify the component means $\mu_1$ and $\mu_2$. We randomly pick $\mu_1$ from $\mathcal{S}^{d-1}$ and then randomly pick $\mu_2 \in\mathcal{S}^{d-1}$ that makes an angle $\pi/4$ with $\mu_1$ (basically, once we have chosen $\mu_1$, we keep on sampling points uniformly at random from $\mathcal{S}^{d-1}$ until we find a point that makes an angle $\pi/4$ with $\mu_1$.) We use two sets of values for $\kappa_1$ and $\kappa_2$. In one set we use $\kappa_1 = 10$ and $\kappa2=5$. In another set we use $\kappa_1 = 80$ and $\kappa_2 = 100$. For larger $\kappa$ values individual von Mises-Fisher distributions becomes narrower around their respective means. We set mixture parameter $w=0.3$.

Once we specify $d, \mu_1,\mu_2\in\mathcal{S}^{d-1}, \kappa_1,\kappa_2>0$ and $w\in (0,1)$ we can easily sample data points from the resulting mixture of von Mises-fisher distribution using \texttt{scipy}. We sample three sets of iid samples from mixture of von Mises-Fisher distribution. The first set contains 10000 samples and is used to train the density estimators (training set). The second set contains 2000 samples and is used to choose hyper-parameters of various estimators (validation set) and the last set contains 10000 samples and is used to test the performance of the density estimators (test set).

We run our experiments on a laptop with Intel Xeon W-10855M Processor, 64GB memory and NVDIAQuadro RTX 5000 Mobile GPU (with 16GB memory).

\subsection{Hyper-parameter selection}

For each of the three estimators KNNDE, KDE and EaSDE, we select hyper-parameters using the validation set. If there are $n$ samples in the training set, then for KNNDE, we choose the best value of $k$, the number of nearest neighbors, from the set $\{1, 10, 50, \log(n), \sqrt{n}, n/8, n/4, n/2, 3n/4\}$ for which empirical total variation distance of KNNDE is smallest on the validation set. For any data point, we find its k nearest neighbors from the training set using \texttt{sklear.neighbors} from \texttt{scikit-learn}.\footnote{\url{https://scikit-learn.org/stable/}}

We implemented KDE using \texttt{scikit-learn}'s \texttt{KernelDensity} package. We used Gaussian kernel and the bandwidth $h$ of the Gaissian kernel was chosen using the validation set. We selected the best $h$ value from the set \texttt{np.logspace(-2, 0, 20)} containing 20 values in the log space from the interval $[10^{-2}, 1]$.

For EasDE, we varied $m$ via varying the expansion factor which is defined as $m/d$. For each fixed $m$, we selected the value of $k$ using the validation set. The various values of $k$ that were considered were from the set $[d\log(m) -8, d\log(m) +8]$.

Eroor bars of EaSDE we computed from 10 independent runs of EaSDE, each time using a different random projection matrix $\Theta$.


\begin{figure}
  \begin{center}
    \includegraphics[scale=0.6]{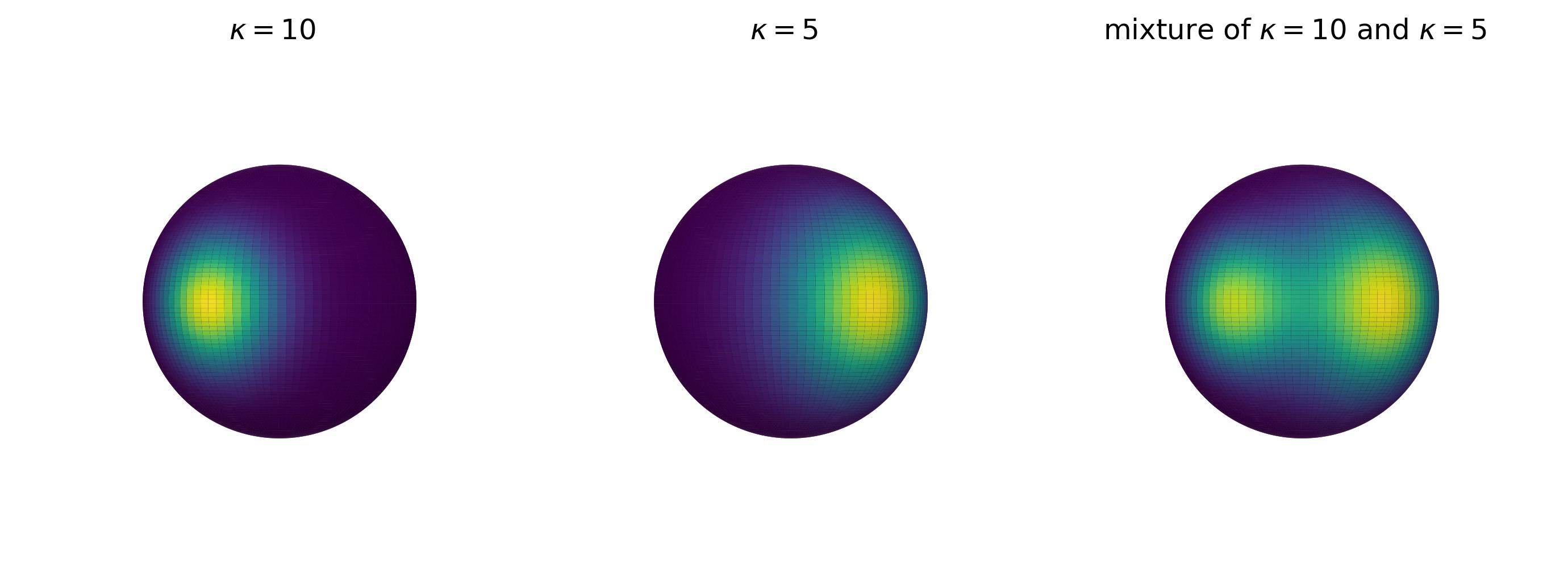}
  \end{center}
  \caption{Individual components and a mixture of von Mises-Fisher distributions on $\mathcal{S}^2$. The left panel shows probability density of von Mises-Fisher distribution with parameters $\mu_1 = \left(-\sqrt{0.9},-\sqrt{0.1},0\right)\in\mathcal{S}^2$ and $\kappa_1 = 10$. The middle panel shows the probability density function of von Mises-Fisher distribution with parameters $\mu_2 = \left(-\sqrt{0.01}, -\sqrt{0.99}, 0\right)\in\mathcal{S}^2$ and $\kappa_2 = 5$. The right panel shows the probability density function of a mixture of these two von Mises-Fisher distributions where the mixing coefficient $w=0.3$. In all these picture bright yellow color represents high density regions and the shifting shades from yellow to green represents gradual low density regions.}
  \label{fig:vonmises}
\end{figure}

\end{document}